\documentclass[a4paper,11pt]{article}

\usepackage{amsmath,amsfonts,enumerate,amssymb,graphicx}
\usepackage{color}
\usepackage[dvips]{epsfig}
\usepackage{graphicx}
\usepackage{mathtools}
\usepackage{caption}
\usepackage{subcaption}
\usepackage{float}
\newtheorem{theorem}{Theorem}
\newtheorem{proposition}{Proposition}

\newtheorem{lemma}{Lemma}
\newtheorem{corollary}{Corollary}
\newtheorem{definition}{Definition}
\newtheorem{example}{Example}

\newcommand{\Qed}{\rule{2.5mm}{3mm}}

\newcounter{case}

\renewcommand{\thecase}{\arabic{case}}

\newcounter{subcase}

\numberwithin{subcase}{case}

\newcounter{subsubcase}

\numberwithin{subsubcase}{subcase}

\newenvironment{proof}{{\noindent \sc Proof.}}{\hfill $\Qed$ \\}

\title{Generalized tableaux over arbitrary digraphs and their associated differential equations}

\author{Luis Mart\'inez\thanks{Luis Mart\'inez is supported by the UPV/EHU and Basque Center of Applied Mathematics, grant US21/27.} University of the Basque Country (UPV/EHU),\\ Department of Mathematics, Spain\\ Antonio Vera L\'opez, University of the Basque Country (UPV/EHU),\\ Department of Mathematics, Spain\\ Antonio Vera P\'erez, PLD Space, Nicol\'as Cop\'ernico 7, Elche, Spain\\ Beatriz Vera P\'erez, Catholic University of Murcia, Murcia, Spain\\ Olga Basova, University of the Basque Country (UPV/EHU),\\ Department of Mathematics, Spain.}

\date{}

\begin{document}
\maketitle

\noindent{\bf Corresponding author:} Luis Mart\'inez\\ University of the Basque Country (UPV/EHU)\\ Faculty of Science and Technology, Department of Mathematics\\ Leioa, Spain, 48940\\ TEL: +34 946012651\\ Email: luis.martinez@ehu.eus

\begin{abstract} We revisit the concepts of acyclic orderings and number of acyclic orderings of acyclic digraphs in terms of dispositions and counters for arbitrary multidigraphs. We prove that when we add a sequence of nested directed paths to a directed graph there is a unique polynomial such that the generatrix function of the family of counters is the product of the polynomial and the exponential function. We give an application, by considering a kind of digraphs arranged in rows introduced by the authors in a previous paper, called dispositional digraphs, in the particular case in which the digraph has two rows, to obtain new families of linear differential equations of small order whose coefficients are polynomials of small degree which admit polynomial solutions. In particular, we obtain a new differential equation associated to Catalan numbers, and the corresponding associated polynomials, which are solution of this differential equation; we term them Catalan differencial equation and Catalan polynomials, respectively. We prove that the Catalan polynomials obtained when we connect the directed path to the second vertex of the lower row of the digraph are orthogonal polynomials for an appropriate weight function. We characterize the digraphs that maximize the counter of connected dispositional digraphs and we find a new differential equation associated to these digraphs. We introduce also dispositions and counters in any multidigraph with non-strict inequalities in the dispositions, and we find new differential equations associated to some of them.
\end{abstract}

\section{Introduction}
Young Tableaux are an important mathematical concept. They have applications to several areas of Combinatorics, since for instance Graph Theory (\cite{B1974},\cite{M2005},\cite{PPS2023}), Matroid Theory (\cite{H2008}), enumeration of matrices with non-negative entries (\cite{KN1970}), longest increasing subsequences (\cite{OP1989}), and tournaments (\cite{Y1980}).

They are applied also in pure Group Theory, for instance in Representation Theory of Groups (\cite{F1900},\cite{K1983},\cite{P1994}).

Their applications are not limited to pure mathematics, and Young Tableaux are used also in other fields of science such as chemistry (\cite{f2001},\cite{l2011}), nuclear physics (\cite{dd2014},\cite{MRA2010},\cite{nsl2005}), Population Genetics (\cite{ES2012}), or robotics (\cite{hck2011}).

In particular, Vera-L\'opez et al. introduced in \cite{VAB2014} a class of tableaux arranged in two rows related to a new generalization of Catalan numbers, and they applied their results to advance in the study of Higman's conjecture. Vera-L\'opez, Mart{\'\i}nez et al. enlarged the family of digraphs to which these tableaux can be associated in \cite{VM2017}, where they introduced dispositional digraphs and their corresponding counters, relating them also to Higman's conjecture.

The motivation to write this paper was threefold. First, we wanted to deep in the study of the mentioned dispositional digraphs, obtaining an structural characterization of the connected ones for which their counter takes the maximum value. Second, we wanted to generalize the kind of digraphs to which these techniques can be applied, admitting any multidigraph and admitting two scenarios, one in which we consider strict inequalities in the calculus of the counters, and another one in which we consider non-strict inequalities. Third, we wanted to associate differential equations to arbitrary digraphs, making the equation depend on the selection of the digraph and of a given vertex, and to give a method to obtain polynomial solutions of these equations.

Specifically, we present in this paper a connection of differential equations with tableaux in which the latest ones are given with the maximum generality, introducing these tableaux over arbitrary multidigraphs, and associating recursively to them a number that we term the counter of the multidigraph. The connection between both areas is established associating a polynomial to each simple digraph with a distinguished vertex, which we call the companion polynomial of the digraph (and consequently of the tableau) at the vertex.

This way we can obtain a library of differential equations with polynomial solutions that are obtained from the counters of the corresponding associated digraphs.

In Section \ref{fc} we first recall some basic definitions in graph theory, and after that we give the general setting of these generalized tableaux and some properties of counters, and we calculate the counter of some of these digraphs. We introduce also staircase digraphs and prove that their counters are maximal among the ones of connected dispositional digraphs.

In Section \ref{cpdigraph} we give our main result, that when we plug a sequence of directed paths to a vertex of the digraph there exists a polynomial, which we call the companion polynomial of the digraph at the vertex, so that the generatrix function of the sequence of counters is the product of the polynomial and the exponential function.

In Section \ref{deacp} we introduce the connection of companion polynomials of digraphs with differential equations, so that we can establish a library of differential equations that can be solved in a discrete way using the counters of directed graphs.

In Section \ref{disptwolines} we study a family of new differential equations corresponding to certain digraphs arranged in two rows, distinguishing among them a subfamily that we call Catalan differential equation, and in Section \ref{sdec} we analyze another family of new differential equations corresponding to extremal dispositional digraphs whose counters take the maximum value.

Finally, in Section \ref{nsdd} we present, again with the maximum generality, tableaux in which the dispositions correspond to non-strict inequalities, and we obtain new differential equations associated to some of their counters.  

\section{Dispositions and counters}\label{fc}
\begin{definition} A directed graph (or, more briefly, a digraph) is a pair $(V,A)$, where $V$ is a set and $A\subseteq X\times X$. The cardinalities of $V$ and $A$ are called the order and size, respectively, of the digraph. The elements of $V$ and $A$ are called vertices and arcs, respectively.
\end{definition}

More generally, if we take a multiset contained in $X\times X$, that is, if we contemplate the possibility of having multiple arcs, we will speak of a multidigraph. Pairs of the form $(x,x)$ are called loops. We will say that a multidigraph is simple if it has no loops and no arcs with multilicity greater than $1$. In this case we say, more briefly, that it is a simple digraph.

\begin{definition} A subdigraph of a digraph $(V,A)$ is a digraph $(V^{\prime},A^{\prime})$ with $V^{\prime}\subseteq V$ and $A^{\prime}\subseteq A$. The subdigraph of $(V,A)$ induced by a subset $V^{\prime}$ of $V$ is the subdigraph $(V^{\prime},A^{\prime})$ with $$A^{\prime}=\{(u,v)\mid u,v\in V^{\prime},(u,v)\in A\}.$$ We will say that a subdigraph of $(V,A)$ is an induced subgraph if it is induced by $V^{\prime}$ for some subset $V^{\prime}$ of $V$.
\end{definition}

\begin{definition} If $\frak G=(V,A)$ is a digraph and $v\in V$, then $\frak G-v$ will denote the subdigraph induced by $V-\{v\}$.
\end{definition}

\begin{definition} The reverse of a digraph $\frak G=(V,A)$ is the digraph $\frak G^r=(V,A^{\prime})$, where $$A^{\prime}=\{(v,u)\mid (u,v)\in A\}.$$
\end{definition}

\begin{definition} A digraph with vertex set $\{v_1,\dots,v_n\}$ and arc set $\{(v_1,v_2),\dots,(v_{n-1},v_n)\}$ is called a directed path of order $n$, and is denoted by $P_n$.
\end{definition}

\begin{definition} A digraph with vertex set $\{v_1,\dots,v_n\}$ and an empty arc set is called an empty digraph of order $n$, and is denoted by $E_n$.
\end{definition}

\begin{definition} A directed closed walk in a digraph $\frak G=(V,A)$ is a sequence of vertices and arcs $$v_1,(v_1,v_2),v_2,\dots,v_{n-1},(v_{n-1},v_n),v_n,(v_n,v_1),v_1$$
\end{definition}

In practice the arcs are understood and only the vertices are indicated, and the final vertex, which is the same as the initial one, is ommitted, and one writes just $v_1,\dots,v_n$ 

Obviously, any vertex $v_i$ can be taken to be the initial vertex of the directed closed walk.

\begin{definition}\label{wclw} Let $\frak G=(V,A)$ be a digraph, $\frak R$ an equivalence relation on $V$, $\overline V$ the quotient set formed by the equivalence classes and, for $v\in V$, let $\overline v$ be the equivalence class represented by $v$. The quotient digraph $\frak G/\frak R$ of $\frak G$ by the relation $\frak R$ is the digraph $(\overline V,\overline A)$, where $$\overline A=\{(\overline u,\overline v)\mid (u,v)\in A\},$$ in which the loops are removed and multiple arcs with multiplicity greater than $1$ are substituted by a single arc.
\end{definition}

\begin{proposition}\label{eredcw} Let $\frak G=(V,A)$ be a digraph. The relation on $V$ defined by $$u\frak R v\text{ if }u\text{ and }v\text{ are in a directed closed walk}$$ is an equivalence relation.
\end{proposition}

\begin{proof} If $v\in V$ then $v$ and $v$ are in the trivial directed closed walk formed just by $v$. If $u,v$ are in a directed closed walk, then also $v,u$ are in that directed closed walk. If $u,v$ are in a directed closed walk $\frak W_1$ and $v,w$ are in a directed closed walk $\frak W_2$, then $u,w$ are in the directed closed walk obtained concatenating both directed closed walks by the vertex $v$.
\end{proof}

If $\frak R$ is as in the previous proposition, it is evident that $\frak G/\frak R$ is a digraph without non-trivial directed closed walks. It is called in the literature the strong component digraph (\cite{BG2010}).

\begin{definition} If $\frak G_1=(V_1,A_1)$ and $\frak G_2=(V_2,A_2)$ are digraphs such that $\vert V_1\cap V_2\vert\leq 1$ then $$\frak G_1+\frak G_2=(V_1\cup V_2,A_1\cup A_2).$$
\end{definition}

Observe that the order of $\frak G_1+\frak G_2$ can be $\vert V_1\vert+\vert V_2\vert$ or $\vert V_1\vert+\vert V_2\vert-1$, depending on if $V_1$ and $V_2$ are disjoint or they share a vertex.

\begin{definition}\label{defdisposition} Let $\frak G=(V,A)$ be a multidigraph of order $n$. A disposition in $\frak G$ is a biyective mapping $f:V\longrightarrow\{1,\dots,n\}$ such that $f(v_1)>f(v_2)$ whenever $(v_1,v_2)\in A$.
\end{definition}

If $\frak G$ is as above, then the set of dispositions will be denoted by $\Sigma(\frak G)$, and its cardinality by $\sigma(\frak G)$. The number $\sigma(\frak G)$ will be called the counter of $\frak G$.

We will call Young tableaux over the multidigraph $\frak G$ to the triplet $$(\frak G,\Sigma(\frak G),\sigma(\frak G)).$$

Note that some authors (and in particular the authors of the present paper in \cite{VM2017}) use the reverse ordering $f(v_1)<f(v_2)$ in the definition of dispositions (in \cite{BG2010} Bang-Jensen and Gutin call them acyclic orderings and mention that they are often called topological sortings. It is worthy to note also that, when viewing an acyclic digraph as a partial order, acyclic orderings of the digraph correspond bijectively to \lq linear extensions\rq\ of the corresponding partial order). We think that Definition \ref{defdisposition} is more intuitive, since the arrow symbol in the arcs resembles the inequality sign. Besides, in the end there is no difference with respect to the value of the counter because, as we will see later in Proposition \ref{countreverse}, the counter of a digraph is the same as the one of its reverse.

If the multidigraph $\frak G=(V,A)$ has a loop at vertex $v$, then no disposition exists, because as $(v,v)\in A$, if $f$ was a disposition we would have $f(v)>f(v)$. Consequently, $\sigma(\frak G)=0$ in this case.

Also, if $a\in A$ is a multiarc with multiplicity $m>1$, then it is obvious that if we replace the $m$ parallel arcs by a single arc the counter does not change.

Hence, from now on we will consider simple digraphs, that is, multidigraphs without loops or multiple arcs (we will omit, by abuse of notation, the term \lq simple\rq, and we will say just \lq digraphs\rq).

\begin{definition} If $\frak G=(V,A)$ is a digraph, we will say that a vertex $u\in V$ is a minimum point of $\frak G$ if it does not exists any $v\in V-\{u\}$ with $(u,v)\in A$. We will denote by $\min(\frak G)$ the set of minimum points of $\frak G$. Similarly, we will say that a vertex $v\in V$ is a maximum point of $\frak G$ if it does not exists any $u\in V-\{v\}$ with $(u,v)\in A$. We will denote by $\max(\frak G)$ the set of maximum points of $\frak G$.
\end{definition}

In the theory of digraphs minimum points and maximum points are usually called sinks and sources, respectively (\cite{BG2010}).

\begin{theorem}\label{rec}
\begin{enumerate}
\item $$\sigma(\frak G)=\sum_{u\in\min(\frak G)}\sigma(\frak G-u).$$
\item $$\sigma(\frak G)=\sum_{v\in\max(\frak G)}\sigma(\frak G-v).$$
\end{enumerate}
\end{theorem}

We will use the following lemma to prove the theorem. Before stating the lemma, we will establish the following notation: if $\frak G=(V,A)$ is a digraph of order $n$, and if $v\in V$ is a maximum point and $f$ is a disposition of $\frak G-\{v\}$, then $f+(v\to n)$ is the mapping from $V$ to $\{1,\dots,n\}$ defined by $(f+(v\to n))(w)=f(w)$ whenever $w\not=v$ and $(f+(v\to n))(v)=n$, and if $D$ is a set of dispositions of $\frak G-\{v\}$, then $$D+(v\to n)=\{f+(v\to n)\mid f\in D\}.$$ We define similarly $f+(u\to 1)$ when $u$ is a minimum point by $(f+(u\to 1))(w)=f(w)+1$ whenever $w\not=u$ and $(f+(u\to 1))(u)=1$, and $$D+(u\to 1)=\{f+(u\to 1)\mid f\in D\}.$$

\begin{lemma}
If $\frak G$ is a digraph of order $n$, then
\begin{enumerate}
\item $$\Sigma(\frak G)=\dot{\bigcup}_{u\in\min(\frak G)}(\Sigma(\frak G-u)+(u\to 1)).$$
\item $$\Sigma(\frak G)=\dot{\bigcup}_{v\in\max(\frak G)}(\Sigma(\frak G-v)+(v\to n)).$$
\end{enumerate}
\end{lemma}

\begin{proof} We will prove it for the set of maximum points, being the proof similar for minimum points. If $f\in \Sigma(\frak G-v)$, then $f+(v\to n)\in\Sigma(\frak G)$ because, as $v$ is a maximum point, $f+(v\to n)$ is a disposition in $\frak G$. Reciprocally, if $f$ is a disposition in $\frak G$, as $f$ is biyective, there is a unique vertex $v$ such that $f(v)=n$, and $v$ must be obviously a maximum point. Besides, $f=f^{\star}+(v\to n)$, where $f^{\star}$ is the restriction of $f$ to $V-\{v\}$. This proves that $$\Sigma(\frak G)=\bigcup_{v\in\max(\frak G)}(\Sigma(\frak G-v)+(v\to n)).$$ It is obvious that the union is disjoint, because if $f\in(\Sigma(\frak G-v_1)+(v_1\to n))\cap(\Sigma(\frak G-v_2)+(v_2\to n))$ with $v_1\not=v_2$, then $f(v_1)=n$ and $f(v_2)=n$, and this contradicts the injectivity of $f$.
\end{proof}

Now we will prove the theorem:

\begin{proof} It is a direct consequence of the lemma (which allows us in passing to obtain all the dispositions in an iterative way), using the fact that the cardinality of a disjoint union of finite sets is the sum of the cardinalities of the sets.
\end{proof}

\begin{proposition} If $\frak G$ is a digraph, then $\sigma(\frak G)>0$ if and only if $\frak G$ has no non-trivial directed closed walk.
\end{proposition}

\begin{proof} If $\frak G$ had a non-trivial directed closed walk $v_1,\dots,v_n$ and $f$ was a disposition we would conclude that $f(v_1)>f(v_1)$, which leads to a contradiction, and thus $\sigma(\frak G)=0$. Conversely, we will prove by induction on the order $n$ of $\frak G$ that if $\frak G$ has no non-trivial directed closed walk, then $\sigma(\frak G)>0$. It is trivially true if $n=1$, because in that case $\sigma(\frak G)=1$. Let us suppose that it is true when the order is $n-1$, and let $n>1$. Since $\frak G$ has no non-trivial directed closed walk, $\min(\frak G)\not=\emptyset$, and now the result follows from Theorem \ref{rec} by using the inductive hypothesis.
\end{proof}

If $\frak G$ has no non-trivial directed closed walk, then the transitive clausure of the relation associated to the graph is a partial order, and the dispositions and counters can be interpreted in terms of $(P,\omega)$-partitions (\cite{ch2023},\cite{S1971}).

\begin{proposition}\label{countersubdigraph} If $\frak G^{\prime}=(V,A^{\prime})$ is a subdigraph of a digraph $\frak G=(V,A)$, then $$\sigma(\frak G)\leq\sigma(\frak G^{\prime}).$$
\end{proposition}

\begin{proof} Let $n$ be the order of $\frak G$. If $f:V\longrightarrow\{1,\dots,n\}$ is a disposition in $\frak G$, then $f$ is also a disposition in $\frak G^{\prime}$.
\end{proof}

\begin{proposition}\label{contpath}$\sigma(P_n)=1\ \forall n\in\Bbb N$.
\end{proposition}

\begin{proof} We will prove it by induction in $n$. It is obvious for $n=1$. Let us suppose that it is true for $n-1$. By Theorem \ref{rec}, $\sigma(P_k)=\sigma(P_{k-1})$, and now the result follows from the inductive hypothesis.
\end{proof}

\begin{proposition}\label{countreverse} If $\frak G$ is a digraph, then $\sigma(\frak G)=\sigma(\frak G^r)$.
\end{proposition}

\begin{proof} It can be immediately proved by induction on the order of $\frak G$, using Theorem \ref{rec} and having into account that $\min(\frak G)=\max(\frak G^r)$ and $\max(\frak G)=\min(\frak G^r)$.
\end{proof}

\begin{theorem}\label{contcc} If $\frak G=(V,A)$ is a digraph and $\frak G_1=(V_1,A_1),\dots,\frak G_r=(V_r,A_r)$ are the connected components of the underlying undirected graph, of orders $n_1,\dots,n_r$, respectively, then $$\sigma(\frak G)=\binom{n_1+\dots+n_r}{n_1,\dots,n_r}\sigma(\frak G_1)\cdots \sigma(\frak G_r).$$
\end{theorem}

\begin{proof} First, let us observe that, if we generalize Definition \ref{defdisposition} changing $\{1,\dots,n\}$ to any set $S$ of $n$ natural numbers, where $n$ is the order of $\frak G$, then the counter $\sigma(\frak G)$ does not change, because Theorem \ref{rec} continues being true. Let us put $S=\{1,\dots,n_1+\dots+n_r\}$. If $f:V\longrightarrow S$ is a disposition in $\frak G$ and if $S_i=f(V_i)$ for $i=1,\dots,r$, then $f_{\vert V_i}$ is a disposition in $\frak G_i$ for every $i$, and the sets $S_1,\dots,S_r$ form a partition of $S$. Reciprocally, if $S_1,\dots,S_r$ is a partition of $S$ with $\vert S_i\vert=n_i\ \forall i\in\{1,\dots,r\}$, and if $f_i:V_i\longrightarrow S_i$ is a disposition in $\frak G_i\ \forall i\in\{1,\dots,r\}$, then it is obvious that the mapping $f:V\longrightarrow S$ defined by $f(v)=f_i(v)$ where $i$ is the only index satisfying that $v\in S_i$, is a disposition in $\frak G$. Now the conclusion of the theorem follows immediately.
\end{proof}

\begin{corollary}\label{counten} $\sigma(E_n)=n!\ \forall n\in\Bbb N$.
\end{corollary}

\begin{proof} The digraph $E_n$ has $n$ connected components, and the counter of each one of them is $1$.
\end{proof}

\begin{definition} Let $S=\{\frak G_i\}_{i\in\Bbb Z^+}$, where $\Bbb Z^+=\{i\in\Bbb Z\mid i\geq 0\}$, be a sequence of digraphs such that $\frak G_i$ is a subdigraph of $\frak G_{i+1}$ for every $i$. We will call generatrix function of $S$ to $$F_S(X)=\sum_{i=0}^{\infty}\frac{\sigma(\frak G_i)}{i!}X^i.$$
\end{definition}

\begin{proposition}\label{generatrixpath} $F_{\{P_{i+1}\}_{i\in\Bbb Z^+}}(X)=\exp(X)$.
\end{proposition}

\begin{proof} Use Proposition \ref{contpath} and the MacLaurin expansion of the exponential function.
\end{proof}

\begin{definition} A tree is a conected acyclic graph. A rooted tree is a tree in which a vertex is selected. This vertex is called the root of the rooted tree.
\end{definition}

We can give a natural orientation to the edges of a rooted tree, where the origin of the arc is the vertex of the edge that is closest to the root, so that we can associate a digraph to the tree. If $\frak T$ is a rooted tree with root $v$, then $\frak T-v$ is the disjoint union of $k$ rooted trees $\frak T_1,\dots,\frak T_k$, where $k$ is the degree of $v$ and where, if $\{v_1,\dots,v_k\}$ is the neighbourhood of $v$ in $\frak T$, then $v_i$ is the root of $\frak T_i$. We will call $\frak T_1,\dots,\frak T_k$ the descendants of $\frak T$. 

\begin{proposition} If $\frak T$ is a rooted tree with $k$ descendants $\frak T_1,\dots,\frak T_k$ of orders $n_1,\dots,n_k$, respectively, then $$\sigma(\frak T)=\binom{n_1+\dots+n_k}{n_1,\dots,n_k}\sigma(\frak T_1)\cdot\dots\cdot\sigma(\frak T_k).$$
\end{proposition}

\begin{proof} Use Theorem \ref{rec} and, after that, Theorem \ref{contcc} with the digraph associated to $\frak T-v$.
\end{proof}

\begin{corollary} Let $q\geq 2$ be a natural number. Let $\frak T^{(q)}$ denote the $q$-ary tree and, for $n\in\Bbb N$, let $\frak T^{(q)}_n$ denote the $n$-level of $\frak T^{(q)}$. Then, $$\sigma(\frak T^{(q)}_1)=1$$ and $$\sigma(\frak T^{(q)}_{n+1})=\binom{\frac{q(q^n-1)}{q-1}}{\frac{q^n-1}{q-1},\dots,\frac{q^n-1}{q-1}}\sigma(\frak T^{(q)}_n)^q$$ for $n\in\Bbb N$.
\end{corollary}

For $q=2$ the numbers obtained in the previous corollary correspond to sequence A056972 in Sloane's on-line encyclopedia of integer sequences (\cite{OEIS}).

For $q=3$ they correspond to sequence A273723.

We introduce now a family of directed graphs, the staircase digraphs:

\begin{definition} Let $n\in\Bbb N$. We define the staircase digraph $\frak S_n$ to be the digraph with vertex set $V_n=\{v_1,\dots,v_n\}$ and arc set $A_n=\{(v_1,v_2),(v_3,v_2),(v_3,v_4),(v_5,v_4),\dots\}=\{(v_{2i+1},v_{2i+2})\mid 0\leq i\leq\lfloor\frac{n-2}{2}\rfloor\}\cup\{(v_{2i-1},v_{2i-2})\mid 2\leq i\leq\lfloor\frac{n+1}{2}\rfloor\}$.
\end{definition}

We will denote by $s_n$ the counter of $\frak S_n$.

We show the staircase digraph $\frak S_5$ in the following figure:

\begin{figure}[H]
\caption{Staircase digraph of order $5$}
\includegraphics[width=5 in]{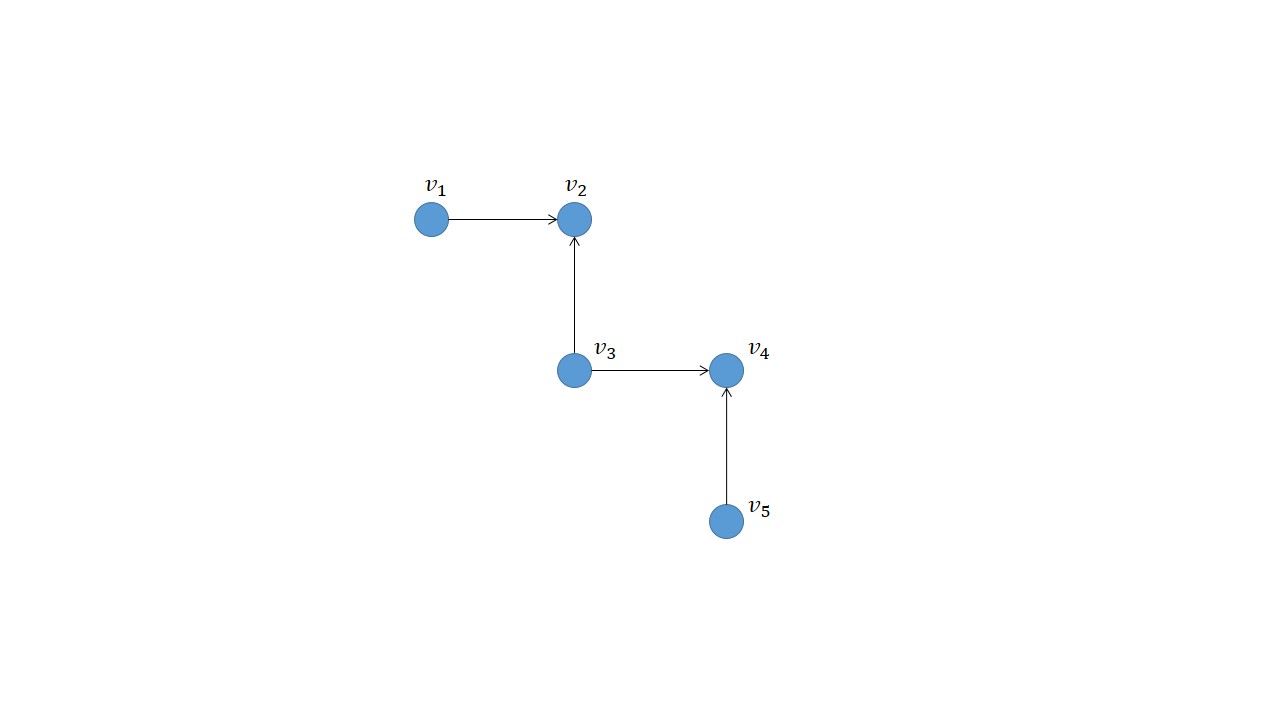}
\end{figure}

The following result holds trivially and its proof will be omitted.

\begin{proposition} $$\min(\frak S_n)=\{v_{2i}\mid 1\leq i\leq\lfloor\frac{n}{2}\rfloor\}=V_n\cap \{v_i\mid i\in 2\Bbb N\},$$
$$\max(\frak S_n)=\{v_{2i+1}\mid 0\leq i\leq\lfloor\frac{n-1}{2}\rfloor\}=V_n\cap \{v_i\mid i\in 1+2\Bbb N\}.$$
\end{proposition}

By convenience we will consider also the case $n=0$. In this case $\frak S_0$ is the void digraph and $\sigma(\frak S_0)=1$, because the only disposition is the void mapping (this may sound strange at first sight, but it will be justified when we study the generating function of this family of digraphs, and it is also a convey for the case $i=\frac{n}{2}$ in the first part of the following theorem and for $i=0$ or $i=\frac{n-1}{2}$ in the second part).

\begin{theorem}\label{recstarcaise} If $n\geq 2$, then $$s_n=\sum_{1\leq i\leq\lfloor\frac{n}{2}\rfloor} {n-1 \choose 2i-1} s_{2i-1}s_{n-2i},$$ and if $n\geq 1$, then
$$s_n=\sum_{0\leq i\leq\lfloor\frac{n-1}{2}\rfloor} {n-1 \choose 2i} s_{2i}s_{n-2i-1}.$$
\end{theorem}

\begin{proof} Use the previous proposition and Theorems \ref{rec} and \ref{contcc}.
\end{proof}

The first 7 values of $s_n$, for $n$ from $0$ to $6$, are $1,1,1,2,5,16,61$.

For arbitrary $n$, they corresponds to sequence A000111 in \cite{OEIS}, and they are the Euler zigzag numbers. It is known by a result of Andr\'e that the sequence gives half the number of alternating permutations (this follows also immediately from the definition of dispositions, because half of the alternating permutations correspond to dispositions in $\frak S_n$, and the other half to dispositions in $\frak S^r_n$). Euler zigzag numbers have also interesting connections with quantum mechanics (\cite{H},\cite{HS2007},\cite{SH2007})

The fact that the counters of staircase digraphs are the Euler zigzag numbers has the following interpretation in terms of generatrix functions:

\begin{theorem} $$F_{\{\frak S_i\}_{i\in\Bbb Z^+}}(X)=\sec(X)+\tan(X)$$
\end{theorem}

By summing the two expressions for $s_n$ obtained in Theorem \ref{recstarcaise} we get the following formula mentioned in \cite{OEIS}:

If $n\geq 1$, then $$2s_{n+1}=\sum_{k=0}^n {n \choose k} s_k s_{n-k}$$

We introduced in \cite{VM2017} dispositional digraphs in the following way:

\begin{definition} Let $m$ be a natural number and consider $n$ non-negative integers $a_1,\dots,a_n$
such that $a_1+\cdots+a_n=m$ and $n-1$ integer numbers $b_2,...,b_n$. We define the dispositional directed graph \begin{equation}\mathcal G([a_1,0],[a_2,b_2],\dots,[a_n,b_n])\end{equation} to be the digraph with $m$ vertices arranged vertically in $n$ rows of lengths $a_1,...,a_n$, respectively, where for each $i>1$ the $i$-th row is shifted $b_i$ places to the right with respect to the previous one if $b_i>0$, shifted $-b_i$ places to the left if $b_i<0$, and no shifted if $b_i=0$. Moreover, the arcs go from each vertex to its neighbour on the right in the same row, if there is any, and to its neighbour below in the next row, if there is any.
\end{definition}

Dispositional digraphs, and their associated counters, include a large number of tableaux, such as, for example, Young tableaux of skew shape (\cite{S1999}), which are the classical Young tableaux, and also Young tableaux of shifted strips with constant width (\cite{ch2023},\cite{S2017},\cite{VM2017}).

Staircase digraphs are dispositional digraphs. In fact, if $n=2k$, then $\frak S_n$ is $$\mathcal G([2,0],[2,-1],\dots,[2,-1]),$$ where $[2,-1]$ is repeated $k-1$ times, and if $n=2k+1$, then $\frak S_n$ is $$\mathcal G([1,0],[2,-1],\dots,[2,-1]),$$ where $[2,-1]$ is repeated $k$ times. Moreover, all of them correspond to Young tableaux of skew shape.

Next, we will prove that the connected dispositional digraphs with maximum counter are precisely the staircase digraphs. This states and solves the extremal directed graph theory problem of characterizing the structure of connected induced subdigraphs of a given order $m$ of the cartesian product of two directed paths with the maximum number of acyclic orderings (we would like to settle also as an open problem the question of determining the connected induced subdigraphs of a given order $m$ of the cartesian product of $k$ directed paths with the maximum number of acyclic orderings for a given $k\geq 3$ and determine the corresponding number of acyclic orderings):

\begin{theorem} If \begin{equation}\frak G=\mathcal G([a_1,0],[a_2,b_2],\dots,[a_n,b_n])\end{equation} is a dispositional digraph of order $m=a_1+\dots+a_n$ such that its underlying undirected graph is connected, then $\sigma(\frak G)\leq s_m$. Besides, the equality holds if and only if $\frak G$ is isomorphic to $\frak S_m$ or to $\frak S^r_m$.
\end{theorem}

We will use the following lemma:

\begin{lemma}\label{sjsisjmi} If $j\in\Bbb Z^+$, then $$s_j\leq {j \choose i} s_is_{j-i}\ \forall i\in\{0,\dots,j\}.$$
\end{lemma}

\begin{proof} Use Theorem \ref{contcc} and Proposition \ref{countersubdigraph} with $\frak S_j$ and the subdigraph obtained removing the $i$-th arc.
\end{proof}

Now we will prove the theorem.

\begin{proof} Let $\frak U(\frak G)$ be the underlying undirected graph associated to $\frak G$. If $v$ is a cut-vertex of $\frak U(\frak G)$ that is also a minimum point of $\frak G$, then its out-degree is $0$, by definition of minimum point, and its in-degree could, in principle, be $1$ or $2$, because $\frak U(\frak G)$ is connected and hence the total degree of $v$ is positive. Let us suppose, for the sake of contradiction, that it is $1$.

If there is an arc $(w,v)$, where $w$ is the preceeding vertex of $v$ in its row, then we have the situation depicted in Figure \ref{indegree1} (a), and we deduce that $\frak U(\frak G)-v$ is connected, what contradicts that $v$ is a cut-vertex.

If there is an arc $(w,v)$, where $w$ is the preceeding vertex of $v$ in its column, then we have the situation shown in Figure \ref{indegree1} (b), and thus $\frak U(\frak G)-v$ is connected and we get again a contradiction.

\begin{figure}[H]
    \centering
    \begin{subfigure}[t]{0.45\textwidth}
        \centering
				\caption{}
        \includegraphics[width=5 in]{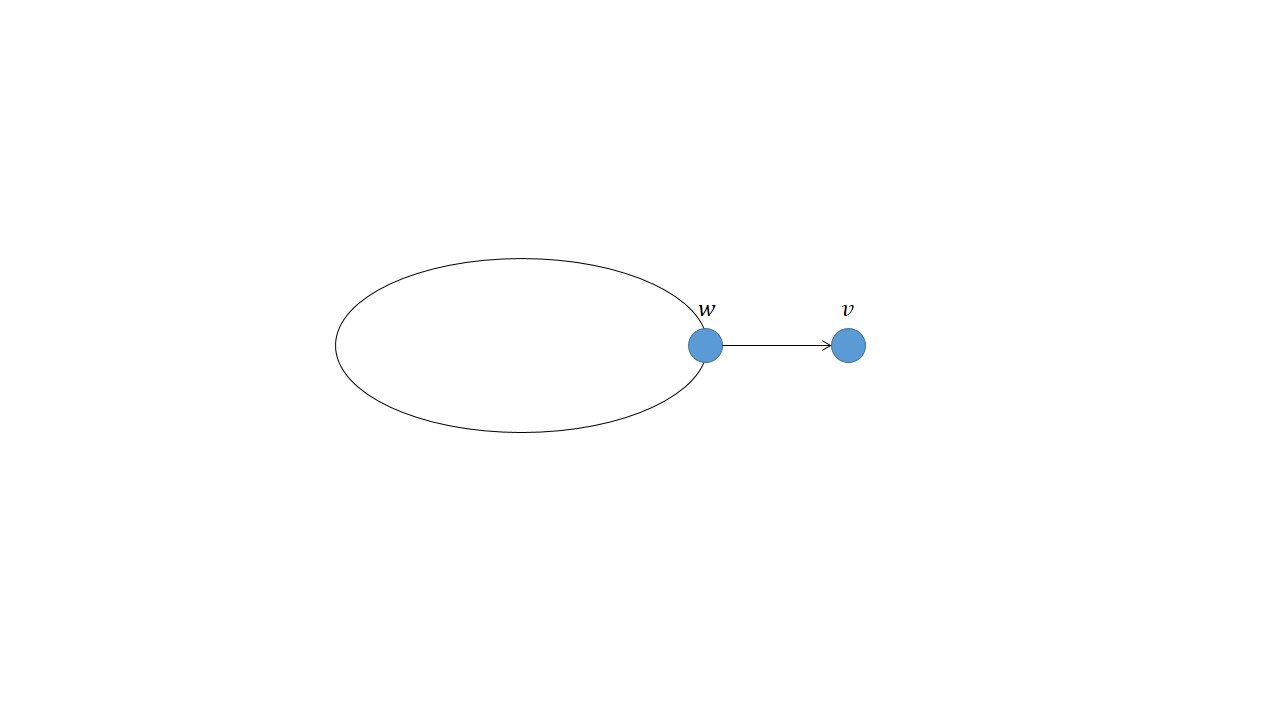} 
    \end{subfigure}
    \hfill
    \begin{subfigure}[t]{0.45\textwidth}
        \centering
				\caption{}
        \includegraphics[width=5 in]{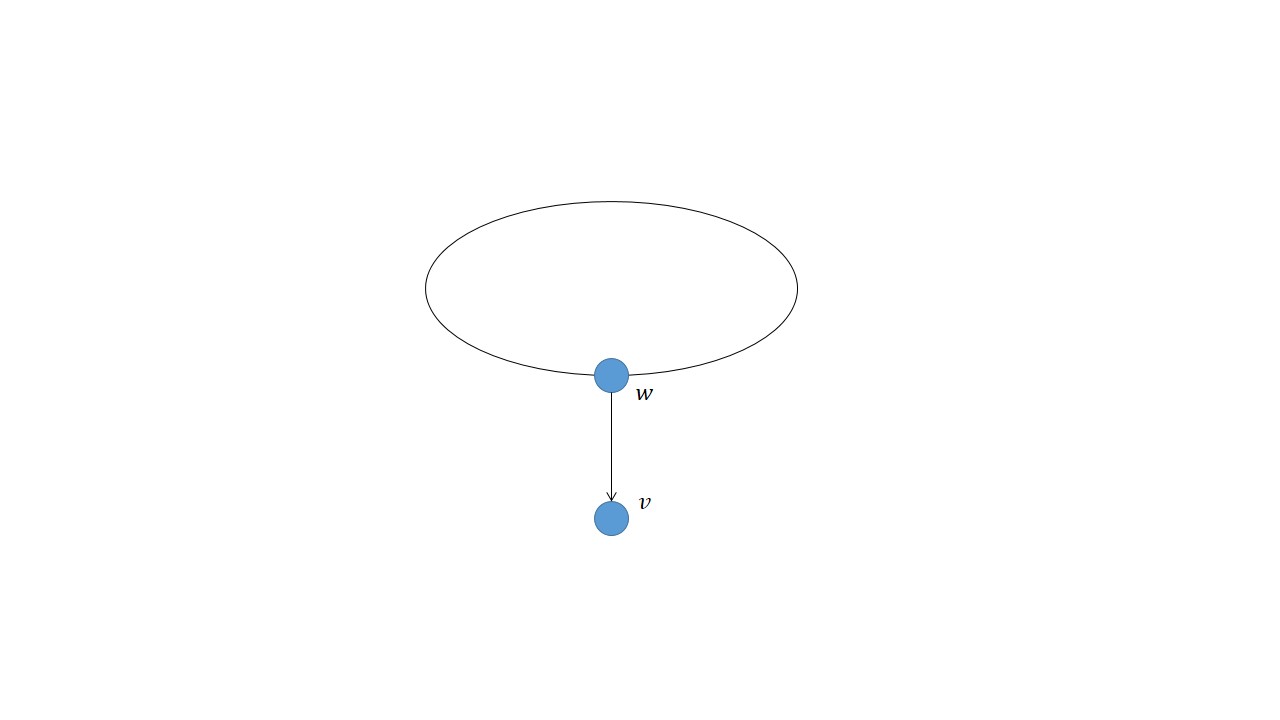} 
    \end{subfigure}
    \caption{In-degree 1. the oval in each part represents the dispositional digraph induced by $V(\frak G)-\{v\}$}\label{indegree1}
\end{figure}

Therefore, the in-degree of $v$ is $2$ and the out-degree of $v$ is $0$, and we are in the situation displayed in Figure \ref{indegree2} (a), so that $\frak U(\frak G)-v$ has two connected components, one of them formed by the vertices in the rows above $v$, and the other one formed by the columns to the left of $v$.

In a similar way, if $v$ is a cut-vertex of $\frak U(\frak G)$ that is also a maximum point of $\frak G$, then we are in the situation depicted in Figure \ref{indegree2} (b), and $\frak U(\frak G)-v$ has two connected components, one of them formed by the vertices in the columns to the right of $v$, and the other one formed by the rows below $v$.

\begin{figure}[H]
    \centering
    \begin{subfigure}[t]{0.45\textwidth}
        \centering
				\caption{}
        \includegraphics[width=4 in]{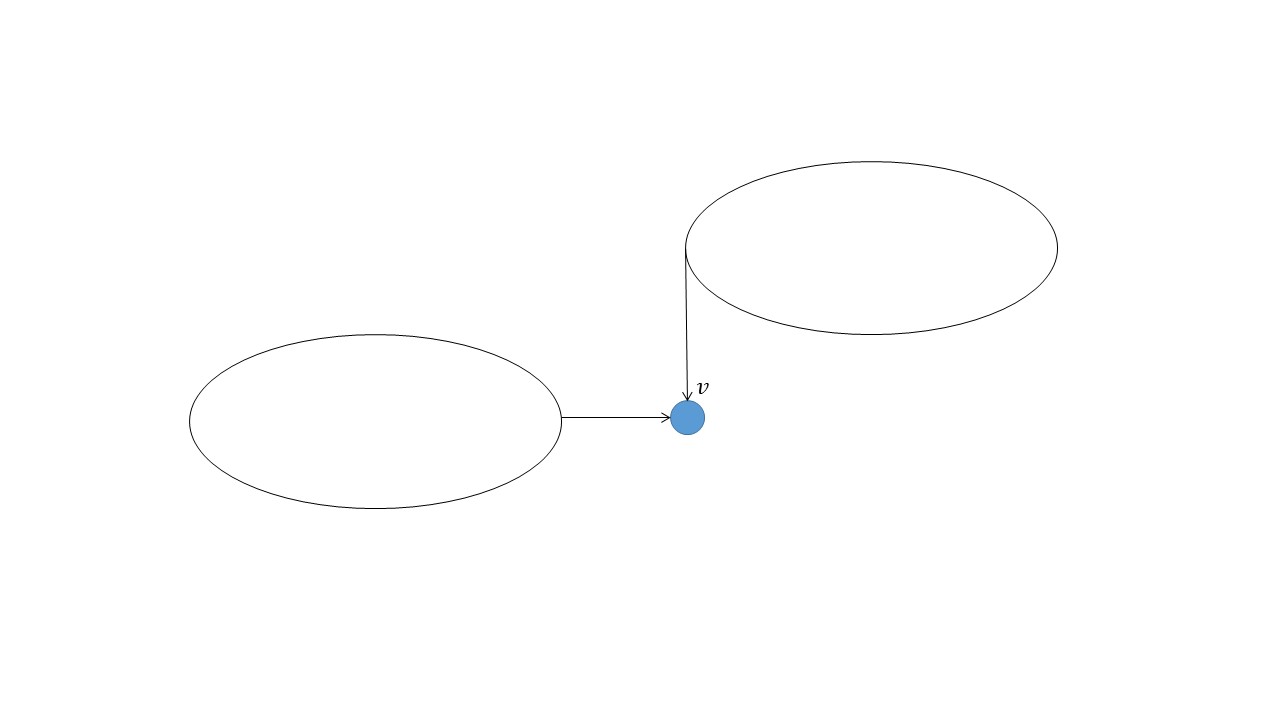} 
    \end{subfigure}
    \hfill
    \begin{subfigure}[t]{0.45\textwidth}
        \centering
				\caption{}
        \includegraphics[width=4 in]{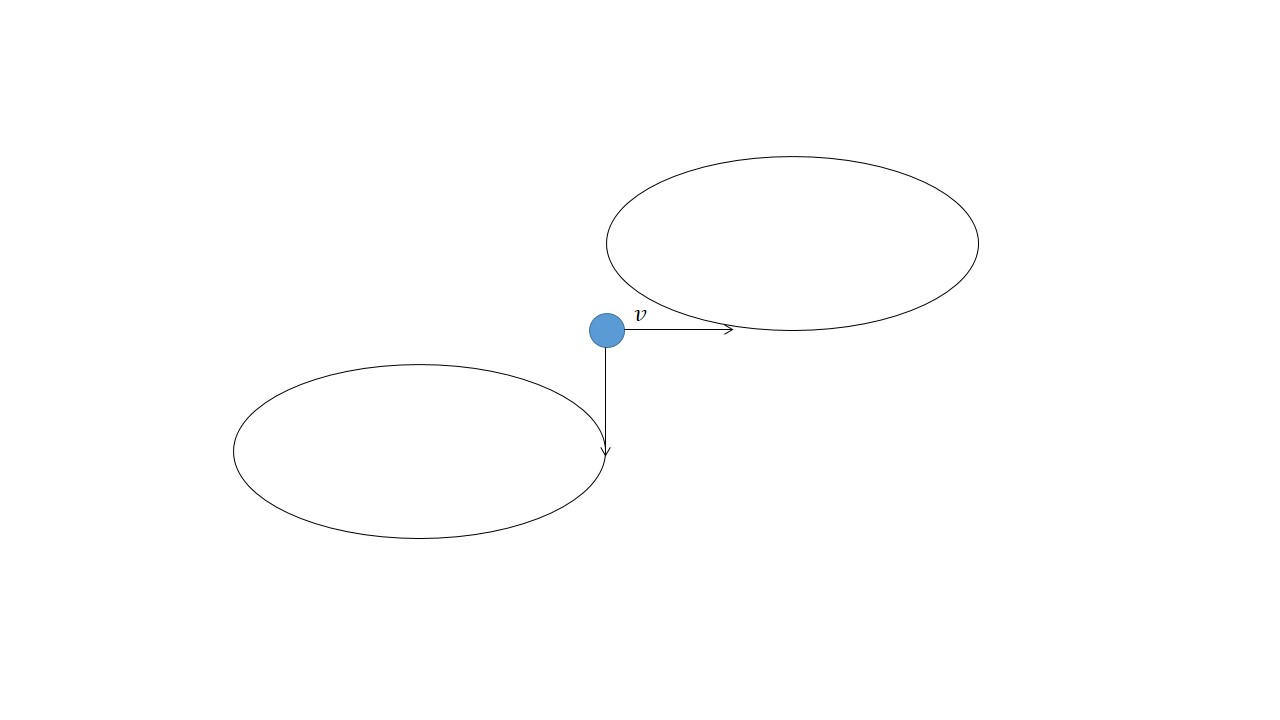} 
    \end{subfigure}
    \caption{In-degree 2. the ovals in each part represent the connected components of $V(\frak G)-\{v\}$}\label{indegree2}
\end{figure}

Let $CV$ be the set of cut-vertices of $\frak U(\frak G)$ and $$S=(\min(\frak G)\cup\max(\frak G))\cap CV.$$ For each one of these vertices exactly one of the two configurations described in Figure \ref{indegree2} holds.

If $v\in S$, we will call $\frak C_1(v)$ to the connected component containing the leftmost vertex in the first row, $\frak C_2(v)$ to the other connected component, and $c(v)$ to $\vert \frak C_1(v)\vert$.

We can order lexicographically the vertices of $S$ so that if $v_1$ and $v_2$ are distinct vertices in $S$, then $v_1\prec v_2$ if $v_1$ is in a row upper than $v_2$ or if $v_1$ and $v_2$ are in the same row but $v_1$ is in a column to the left of $v_2$ (of course, there can be at most two vertices of $S$ in the same row). Obviously, if $v_1\prec v_2$, then $c(v_1)<c(v_2)$.

Let $N=\{c(v)\mid v\in S\}$.

By Theorem \ref{rec}, $$2\sigma(\frak G)=\sum_{u\in\min(\frak G)}\sigma(\frak G-u)+\sum_{v\in\max(\frak G)}\sigma(\frak G-v).$$ This equals $$\sum_{u\in S}\sigma(\frak G-u)+\sum_{v\in(\min(\frak G)\cup\max(\frak G))-CV}\sigma(\frak G-v).$$ By Theorem \ref{contcc}, the first summand equals $$\sum_{u\in S} {m-1 \choose c(u)}\sigma(\frak C_1(u))\sigma(\frak C_2(u)),$$ and this is, by the inductive hypothesis, at most $$\sum_{n\in N} {m-1\choose n} s_n s_{m-1-n}.$$ Now, by the inductive hypothesis, $$2\sigma(\frak G)\leq\sum_{n\in N} {m-1\choose n} s_n s_{m-1-n}+\vert (\min(\frak G)\cup\max(\frak G))-CV\vert s_{m-1}$$ and, by Lemma \ref{sjsisjmi}, if $\{1,\dots,m\}-N=\{n_1,\dots,n_k\}$, then $$2\sigma(\frak G)\leq\sum_{n\in N} {m-1\choose n} s_n s_{m-1-n}+\sum_{i=1}^{\vert (\min(\frak G)\cup\max(\frak G))-S\vert} {m-1 \choose n_i} s_{n_i}s_{m-1-n_i},$$ and this is at most $$\sum_{i=0}^{m-1} {n \choose i} s_is_{m-1-i},$$ and the equality holds if and only if $\min(\frak G)\cup\max(\frak G)=\{1,\dots,m\}$. In this case maximum and minimum points alternate and go trough all the vertices of $\frak G$, and $\frak G$ is isomorphic to $\frak S_m$ or to $\frak S^r_m$.
\end{proof}

The condition that the underlying undirected graph is connected is essential and cannot be avoided. For instance, if we take the empty digraph $E_5$, which is dispositional and not connected then $\sigma(E_5)=120$, but $s_5=16$. In fact, it is obvious, by Proposition \ref{countersubdigraph}, that the maximum value that the counter of a digraph of order $n$ can have is $n!$, and it is clear that the equality holds if and only if the digraph is isomorphic to $E_n$.

The condition that the digraph is dispositional is also esential. If we take the star digraph shown in the following figure,

\begin{figure}[H]
\caption{Star digraph of order $5$}
\includegraphics[width=5 in]{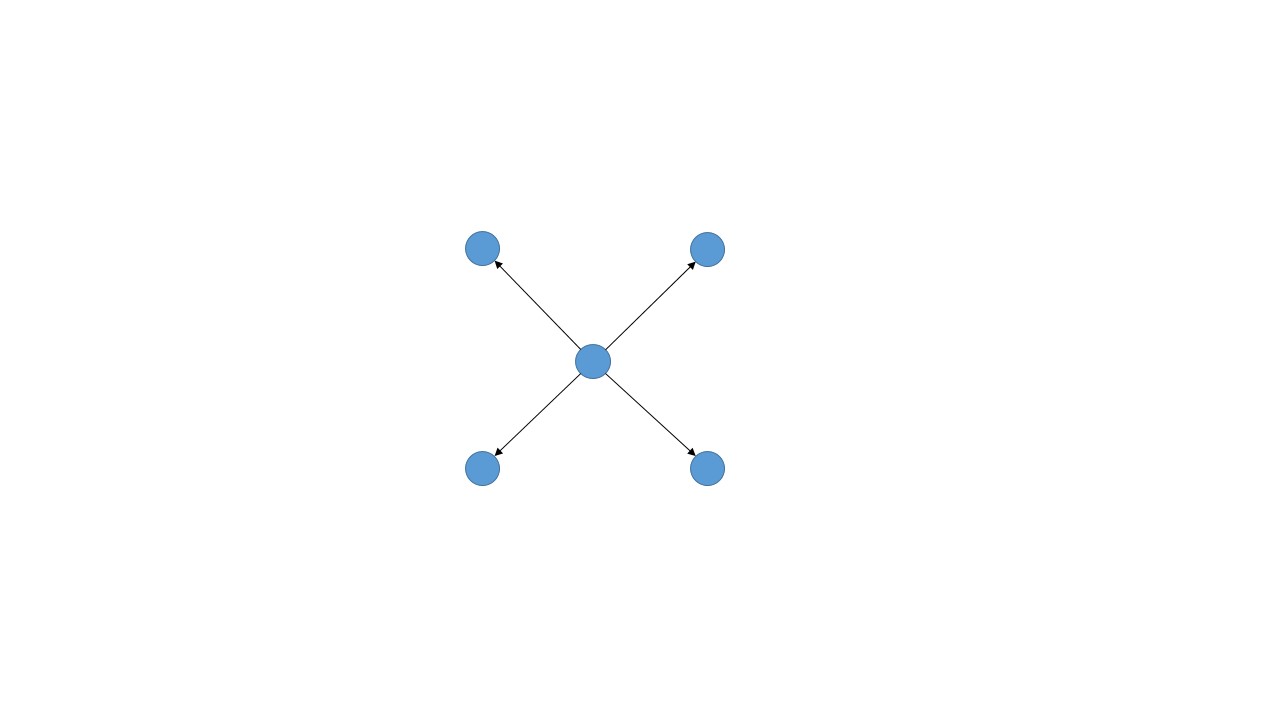}
\end{figure}

which is connected but is not dispositional, then by part 2 of Theorem \ref{rec} we obtain that its counter is $4!=24$, but $s_5=16$.

\section{Companion polynomial of a digraph}\label{cpdigraph}
Before stating the main theorem of this section we will establish some notation. Given a digraph $\frak G=(V,A)$ of order $n$, a vertex $v$ of $V$ and a nonnegative integer $i$, we take $i+1$ vertices $z_0,z_1,\dots,z_i$ with $z_0=v$ and $V\cap\{z_1,\dots,z_i\}=\emptyset$, and we define $P_{i+1}$ to be the directed path $\{z_0,z_1,\dots,z_i\}$ with $i+1$ vertices and arcs $(z_0,z_1),\dots,(z_{i-1},z_i)$, and we also define $\frak G_i=\frak G+P_{i+1}$. Note that when $i=0$ the digraph $\frak G_0$ is just $\frak G$.

In this way we get a head, which is formed by the digraph $\frak G$, and a tail, which is the infinite path with vertices $z_0,z_1,\dots$ whose induced graph on $\{z_0,\dots,z_i\}$ is $P_{i+1}$, as is shown in the next two figures:

\begin{figure}[H]
\caption{Digraph $\frak G$}
\includegraphics[width=5 in]{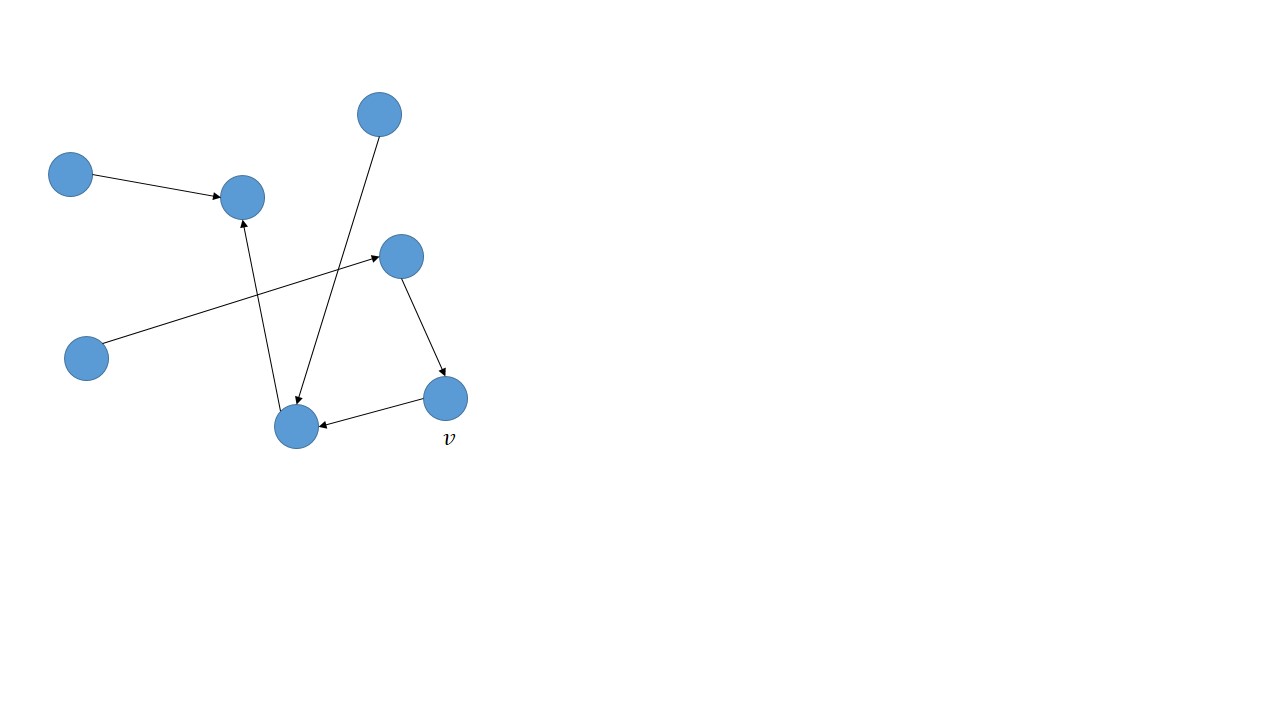}
\end{figure}

\begin{figure}[H]
\caption{Digraph $\frak G$ with infinite path adjoined}
\includegraphics[width=5 in]{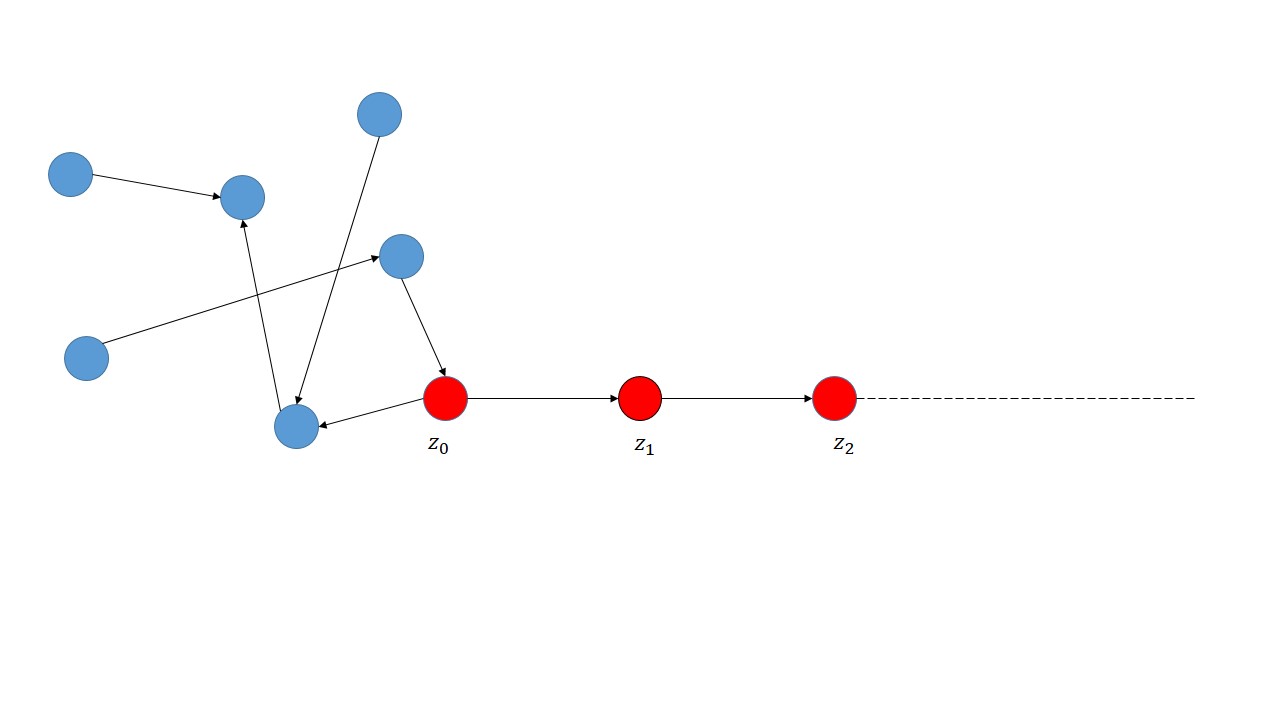}
\end{figure}

\begin{theorem}\label{rcp} Let $\frak G=(V,A)$ be a digraph of order $n$, and let $v\in V$. Then there exists a unique polynomial $T_{\frak G,v}(X)\in\Bbb Q[X]$ with $\deg(T_{\frak G,v}(X))\leq n-1$ such that $$F_{\{\frak G_i\}_{i\in\Bbb Z^+}}(X)=T_{\frak G,v}(X)\exp(X).$$ Besides, the polynomials $T_{\frak G,v}(X)$ satisfy the recurrence relation $$T_{P_1,v}(X)=1$$ and $$\frac{d}{dX}T_{\frak G,v}(X)=\sum_{w\in\min(\frak G)-\{v\}}(T_{\frak G-w,v}(X)+\frac{d}{dX}T_{\frak G-w,v}(X))$$ with $$T_{\frak G,v}(0)=\sigma(\frak G)$$
\end{theorem}

\begin{proof} First, observe that \begin{equation}\label{mingsi}\min(\frak G_i)=S\cup(\min(\frak G)-\{v\})\ \forall i\in\Bbb Z^+, \end{equation} where $$S=\begin{cases} \emptyset, & \text{if }i=0\text{ and }v\not\in\min (\frak G),\\ z_i, & \text{in other case}.\\ \end{cases}$$ We will prove the theorem by induction on the order $n$ of $\frak G$. When $n=1,\frak G=P_1$ and $v$ is the only vertex of the digraph, and it follows from Proposition \ref{generatrixpath} that the result holds with $T_{P_1,v}(X)=1$.

Let us suppose that what we want to prove is true for digraphs of order $n-1$, and let us consider a digraph $\frak G$ of order $n\geq 2$. If $\min(\frak G)-\{v\}=\emptyset$, then $$\sigma(\frak G_i)=\sigma(\frak G)\ \forall i\in\Bbb Z^+,$$ and therefore $$F_{\{\frak G_i\}_{i\in\Bbb Z^+}}(X)=\sum_{i=0}^{\infty}\frac{\sigma(\frak G)}{i!}X^i=\sigma(\frak G)\sum_{i=0}^{\infty}\frac{X^i}{i!}=\sigma(G)\exp (X).$$ If $\min(\frak G)-\{v\}\not=\emptyset$ then, since the independent term of the Maclaurin expansion of $\exp(X)$ is non-zero, there exists a unique formal series $$T_{\frak G,v}(X)=\sum_{i\in\Bbb Z^+}t_i X^i$$ in $\Bbb Q[X]$ such that \begin{equation}\label{mclgi}F_{\{\frak G_i\}_{i\in\Bbb Z^+}}(X)=T_{\frak G,v}(X)\exp(X).\end{equation} We will prove that in the end $T_{\frak G,v}(X)$ is a polynomial of degree at most $n-1$.

Using (\ref{mingsi}) and part 1 of Theorem \ref{rec} we obtain that $$\sigma(\frak G_j)=\sigma(\frak G_{j-1})+\sum_{w\in\min(\frak G)-\{v\}} \sigma((\frak G-w)+P_{j+1})\forall j\in\Bbb Z^+,$$ where we convey that $$\sigma(\frak G_{-1})=\begin{cases} \sigma(\frak G-\{v\}), & \text{if }v\in\min(\frak G),\\ 0, & \text{in other case}.\\ \end{cases}$$ If $i$ is a non-negative integer and $0\leq j\leq i$, by multiplying by $\frac{X^j}{j!}$ in this identity and summing over $j$ we get \begin{equation}\label{recvinmin}\sum_{j=0}^i\frac{\sigma(\frak G_j)}{j!}X^j=\sigma(\frak G_{-1})+\sum_{j=1}^i\frac{\sigma(\frak G_{j-1})}{j!}X^j+\sum_{w\in\min(\frak G)-\{v\}}\sum_{j=0}^i\frac{\sigma((\frak G-w)+P_{j+1})}{j!}X^j.\end{equation}

It follows from (\ref{recvinmin}) that \begin{equation}\label{recvinmin2}F_{\{\frak G_i\}_{i\in\Bbb Z^+}}(X)=\sigma(\frak G_{-1})+\sum_{i\in\Bbb Z^+}\frac{\sigma(\frak G_i)}{(i+1)!}X^{i+1}+\sum_{w\in\min(\frak G)-\{v\}}F_{\{(G-w)+P_{i+1}\}_{i\in\Bbb Z^+}}(X).\end{equation}

Obviously, \begin{equation}\label{dsumsgifip1}\frac{d}{dX}\sum_{i\in\Bbb Z^+}\frac{\sigma(\frak G_i)}{(i+1)!}X^{i+1}=F_{\{\frak G_i\}_{i\in\Bbb Z^+}}(X)\end{equation} and, by the inductive hypothesis, for each $w\in\min(\frak G)-\{v\}$, the identity \begin{equation}\label{fgmwppipo}F_{\{(\frak G-w)+P_{i+1}\}_{i\in\Bbb Z^+}}(X)=T_{\frak G-w,v}(X)\exp(X)\end{equation} holds, where $T_{\frak G-w,v}(X)$ is a polynomial of degree at most $n-2$. Deriving in (\ref{recvinmin2}) and using (\ref{dsumsgifip1}) and (\ref{fgmwppipo}), we get $$\frac{d}{dX} F_{\{\frak G_i\}_{i\in\Bbb Z^+}}(X)=F_{\{\frak G_i\}_{i\in\Bbb Z^+}}(X)+\sum_{w\in\min(\frak G)-\{v\}}(T_{\frak G-w,v}(X)+\frac{d}{dX} T_{\frak G-w,v}(X))\exp(X)$$ and, by (\ref{mclgi}), we obtain \begin{equation}\label{recvinmin3} \frac{d}{dX}F_{\{\frak G_i\}_{i\in\Bbb Z^+}}(X)=(T_{\frak G,v}(X)+\sum_{w\in\min(\frak G)-\{v\}}(T_{\frak G-w,v}(X)+\frac{d}{dX}T_{\frak G-w,v}(X)))\exp(X).\end{equation}

On the other hand, deriving in (\ref{mclgi}) we have \begin{equation}\label{derivegeneratrix}\frac{d}{dX} F_{\{\frak G_i\}_{i\in\Bbb Z^+}}(X)=(T_{\frak G,v}(X)+\frac{d}{dX} T_{\frak G,v}(X))\exp(X),\end{equation} and we conclude from (\ref{recvinmin3}) and (\ref{derivegeneratrix}), by simplifying the summand $T_{\frak G,v}(X)\exp(X)$ and then simplifying the factor $\exp(X)$  that $$\frac{d}{dX} T_{\frak G,v}(X)=\sum_{w\in\min(\frak G)-\{v\}}(T_{\frak G-w,v}(X)+\frac{d}{dX} T_{\frak G-w,v}(X)).$$

It holds that $$T_{\frak G,v}(0)=\sigma(\frak G_0)=\sigma(\frak G).$$

Finally, observe that $\deg(T_{\frak G,v}(X))\leq n-1$, because for each $w\in\min(\frak G)-\{v\}$ we have that $\deg(T_{\frak G-w,v}(X)+\frac{d}{dX} T_{\frak G-w,v}(X))\leq n-2$, and therefore the degree of its primitives is at most $n-1$.
\end{proof}

We will call to $T_{\frak G,v}(X)$ the companion polynomial of $\frak G$ with respect to the vertex $v$.

We can observe that whenever $v\in\min(\frak G)$ the sum in the recurrence equation goes over $\min(\frak G)-\{v\}$, while when $v\not\in\min(\frak G)$ it goes over $\min(\frak G)$.

It is important to stress that the companion polynomial not only depends on the digraph, but also depends on the vertex of the digraph. For instance, if we take the directed path $P_2$ with vertex set $\{v_1,v_2\}$ and arc set $\{(v_1,v_2)\}$, then $T_{P_2,v1}=X+1$, but $T_{P_2,v2}=1$. This same example also shows that if $\frak G$ is a digraph of order $n$, then the degree of $T_{\frak G,v}(X)$ can be strictly less than $n-1$.

A similar analysis can be done plugging the directed path in reverse order, that is, if for any nonnegative integer $i$ we take $i+1$ vertices $z_0,z_1,\dots,z_i$ with $z_0=v$ and $V\cap\{z_1,\dots,z_i\}=\emptyset$ and we consider the directed path  $P^{\star}_{i+1}$ with $i+1$ vertices $\{z_0,z_1,\dots,z_i\}$ and arcs $(z_i,z_{i-1}),\dots,(z_1,z_0)$, we can define the digraph $\frak G^{\star}_i=\frak G+P^{\star}_{i+1}$, in the way shown in the following two figures:

\begin{figure}[H]
\caption{Digraph $\frak G$}
\includegraphics[width=5 in]{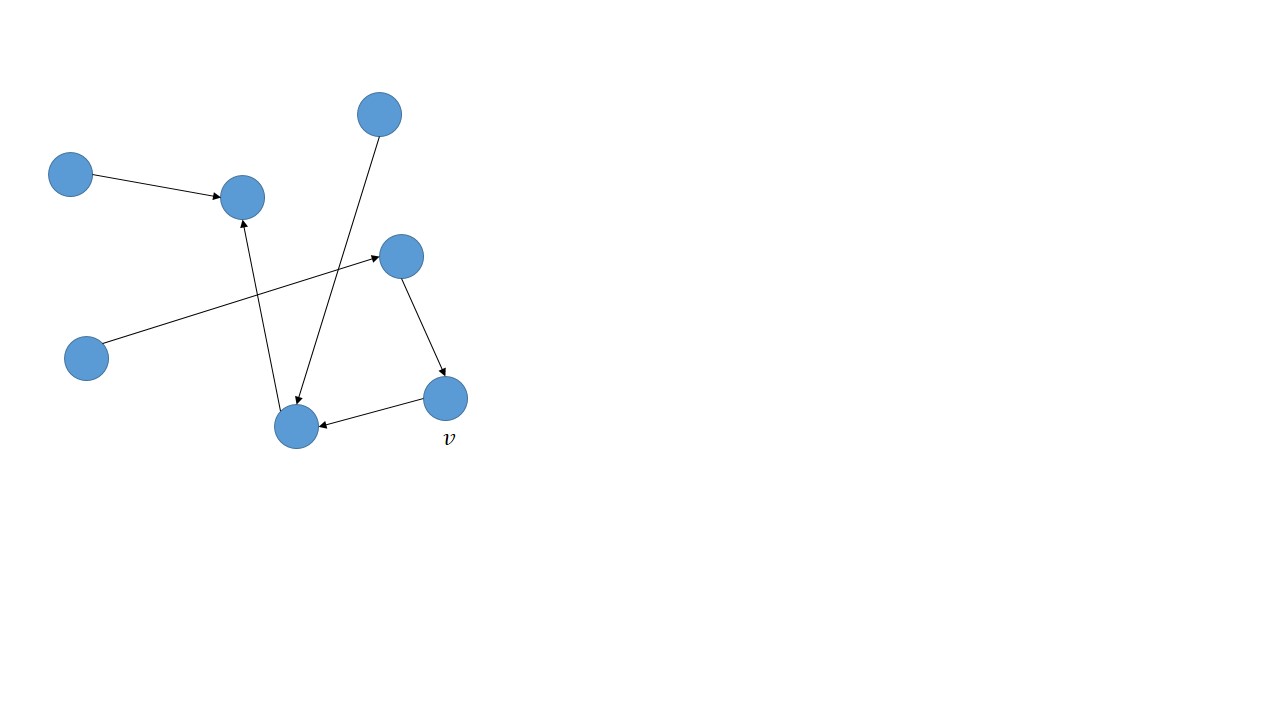}
\end{figure}

\begin{figure}[H]
\caption{Digraph $\frak G$ with infinite path adjoined in reverse order}
\includegraphics[width=5 in]{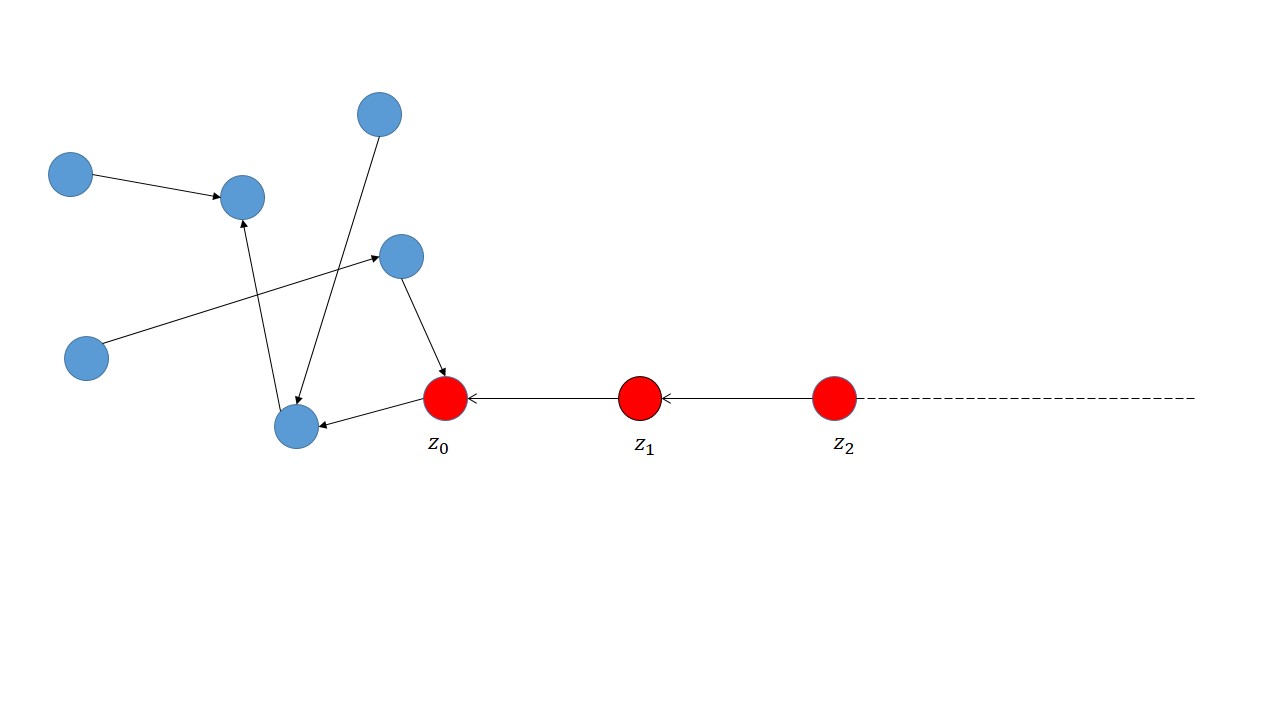}
\end{figure}

Now we have the following theorem, dual of Theorem \ref{rcp}:

\begin{theorem} Let $\frak G=(V,A)$ be a digraph of order $n$, and let $v\in V$. Then there exist a unique polynomial $T^{\star}_{\frak G,v}(X)\in\Bbb Q[X]$ with $\deg(T^{\star}_{\frak G,v}(X))\leq n-1$ such that $$F_{\{\frak G^{\star}_i\}_{i\in\Bbb Z^+}}(X)=T^{\star}_{\frak G,v}(X)\exp(X).$$ Besides, the polynomials $T^{\star}_{\frak G,v}(X)$ satisfy the recurrence relation $$T^{\star}_{P_1,v}(X)=1$$ and $$\frac{d}{dX} T^{\star}_{\frak G,v}(X)=\sum_{w\in\max(\frak G)-\{v\}}(T^{\star}_{\frak G-w,v}(X)+\frac{d}{dX} T^{\star}_{\frak G-w,v}(X))$$ with $$T^{\star}_{\frak G,v}(0)=\sigma(\frak G)$$
\end{theorem}

\begin{proof} Use Theorem \ref{rcp} with $\frak G^r$ and Proposition \ref{countreverse}, and the fact that $\min(\frak G^r)=\max(\frak G)$.
\end{proof}

\section{Some differential equations associated to companion polynomials of digraphs}\label{deacp}

We will begin establishing the fundamental connection of companion polynomials of digraphs with Laguerre equation:

We will consider, for each $j\in\Bbb Z^+$, the directed path $P_{j+1}$ with vertex set $\{v_1,\dots,v_{j+1}\}$ and arc set $\{(v_1,v_2),\dots,(v_j,v_{j+1})\}$. If we consider the companion polynomial $T_j(X)=T(P_{j+1},v_1)(X)$ then we will obtain a recurrence relation for the polynomials $T_j(X)$ in the following proposition:

\begin{proposition}\label{cpolpi}
$$T_0(X)=1$$ and, for $j\geq 1$,
$$\frac{d}{dX}T_j(X)=\frac{d}{dX}T_{j-1}(X)+T_{j-1}(X)\text{ and }T_j(0)=1.$$
\end{proposition}

\begin{proof} By induction on $i$. The base of the induction is obvious, and the step is proved by using Theorem \ref{rcp}, having into account that $\min(P_{j+1})=\{v_{j+1}\}$ and $P_{j+1}-v_{j+1}=P_j$.
\end{proof}

\begin{theorem}\label{jtlp} $T_j(X)=L_j(-X)\forall j\in\Bbb Z^+$, where $L_j$ is the $j$-th Laguerre polynomial.
\end{theorem}

\begin{proof} It is well-known that Laguerre polynomials form a Sheffer sequence (\cite[Chapter 16]{M1992}) in which \begin{equation}\label{reclaguerre}L_0(X)=1\text{ and }\frac{d}{dX}L_j(X)=\frac{d}{dX}L_{j-1}(X)-L_{j-1}(X)\forall j\geq 1, L_j(0)=1\ \forall j\geq 1.\end{equation} Since $$T_0(X)=L_0(-X)=1$$ and also $$\frac{d}{dX}(L_j(-X))=-(\frac{d}{dX}L_j)(-X)$$ and $$\frac{d}{dX}(L_{j-1}(-X))=-(\frac{d}{dX}L_{j-1})(-X),$$ the recurrence relation of Proposition \ref{cpolpi} is the same that the corresponding one in (\ref{reclaguerre}).
\end{proof}

Laguerre polynomials are solutions of Laguerre's equation $$xy^{\prime\prime}+(1-x)y^{\prime}+ny=0.$$ This equation has nonsingular solutions only if $n$ is a non-negative integer. They are orthogonal polynomials with respect to the inner product $$<f,g>=\int_0^{\infty} f(x)g(x) exp(-x)\ dx.$$ Laguerre polynomials have important applications in quantum mechanics, where they are used in the description of the static Wigner functions of oscillator systems, in the 3D isotropic harmonic oscillator and in Morse potential, as well as in the Schrodinger equation for the hydrogen-like atom.

We will give now an alternative characterization of Laguerre polynomials. We will call it the tabloidal form of $L_j(X)$:

\begin{theorem}\label{tfp} $$L_j(X)=(\sum_{i=0}^{\infty}\frac{\binom{i+j}{i}}{i!}(-X)^i)\exp(X)\forall j\in\Bbb Z^+.$$
\end{theorem}

\begin{proof} Following the notation before Theorem \ref{rcp}, if $j\in\Bbb Z^+$, $P_{j+1}$ is the directed path on the vertices $v_1,\dots,v_{j+1}$ and $P_{i+1}$ is the directed path on the vertices $v_1,z_1,\dots,z_i$, it follows from part 2 of Theorem \ref{rec} that $\sigma(P_{j+1}+P_{i+1})=\sigma((P_{j+1}-\{v_1\})+(P_{i+1}-\{v_1\}))$, and since $(P_{j+1}-\{v_1\})+(P_{i+1}-\{v_1\})$ has two connected components which are paths of lengths $j$ and $i$, respectively, then we deduce from Theorem \ref{contcc}, using that the counter of a directed path is $1$, that $$\sigma(P_{j+1}+P_{i+1})=\binom{i+j}{i}.$$ Now, we deduce from Theorem \ref{rcp} that $$\sum_{i=0}^{\infty}\frac{\binom{i+j}{i}}{i!}X^i=T_j(X)exp(X),$$ and the result follows immediately from this identity and Theorem \ref{jtlp}.
\end{proof}

The directed path is not the only digraph whose companion polynomial satisfies Laguerre's equation:

\begin{theorem} If $\frak G$ is a digraph of order $n$ associated to a rooted tree with root $v$, then $$T_{\frak G,v}(X)=\sigma(\frak G)L_{n-1}(-X).$$
\end{theorem}

\begin{proof} We will prove it by induction on $n$. It is obvious if $n=1$. Let us suppose that $n>1$ and that it is true for $n-1$. Then, $\min(\frak G)-\{v\}=\min(\frak G)$, and for any $w\in\min(\frak G)$ the digraph $G-w$ is associated to a rooted tree with root $v$.

Now, by Theorem \ref{rcp} and the induction hypothesis we have $$\frac{d}{dX}T_{\frak G,v}(X)=\sum_{w\in\min(\frak G)}(\sigma(\frak G-w)L_{n-2}(-X)+\sigma(\frak G-w)\frac{d}{dX}L_{n-2}(-X)).$$ By Theorem \ref{rec}, this equals$$\sigma(\frak G)L_{n-2}(-X)+\sigma(\frak G)\frac{d}{dX} L_{n-2}(-X).$$ Factoring out $\sigma(\frak G)$, we obtain $$\sigma(\frak G)(L_{n-2}(-X)+\frac{d}{dX}L_{n-2}(-X))$$ and, by Proposition \ref{cpolpi}, this is $$\sigma(\frak G)\frac{d}{dX}L_{n-1}(-X).$$ Finally, the desired equality follows from the fact that both polynomials in the statement of this theorem have the same independent term.
\end{proof}

Now we will derive from empty digraphs a particular case of generalized Laguerre equation.

We will consider, for each $j\in\Bbb Z^+$, the empty graph $E_{j+1}$ with vertex set $\{v_1,\dots,v_{j+1}\}$ and an empty arc set. If we consider the companion polynomial $T_j(X)=T(E_{j+1},v_1)(X)$ then we will obtain a recurrence relation for the polynomials $T_j(X)$ in the following proposition:

\begin{proposition}\label{cpolpi2}
$$T_0(X)=1$$ and, for $j\geq 1$,
$$\frac{d}{dX}T_j(X)=j(\frac{d}{dX}T_{j-1}(X)+T_{j-1}(X))\text{ and }T_j(0)=(j+1)!.$$
\end{proposition}

\begin{proof} By induction on $j$. The base of the induction is obvious, and the step is proved by using Theorem \ref{rcp}, having into account that $\min(E_{j+1})=\{v_1,\dots,v_{j+1}\}$ and that for $k\in\{2,\dots,j+1\}$ it holds that $E_{j+1}-v_k=E_j$ (with a different labeling of the vertices, of course), and using also Corollary \ref{counten}.
\end{proof}

In what follows we will put for every $j\in\Bbb Z^+$ \begin{equation}\label{coeftj}T_j(X)=\sum_k a_{j,k} X^k,\end{equation} and we will convey that $a_{j,k}=0$ when $k<0$ and when $k>j$.

\begin{lemma}\label{recurrcoefs} $$k a_{j,k}=j(k a_{j-1,k}+a_{j-1,k-1})\ \forall j\in\Bbb N,\forall k\in\Bbb Z.$$
\end{lemma}

\begin{proof} We have by Proposition \ref{cpolpi2} that \begin{equation}\label{rtpjtpjmutjmu}\frac{d}{dX}T_j(X)=j(\frac{d}{dX}T_{j-1}(X)+T_{j-1}(X)).\end{equation} We deduce from (\ref{coeftj}) that \begin{equation}\label{coeftjmo}T_{j-1}(X)=\sum_k a_{j-1,k} X^k.\end{equation} By taking derivatives in (\ref{coeftj}) and in (\ref{coeftjmo}) we obtain \begin{equation}\label{dcoeftj}\frac{d}{dX}T_j(X)=\sum_k k a_{j,k} X^{k-1}=\sum_k (k+1) a_{j,k+1} X^k\end{equation} and \begin{equation}\label{dcoeftjmo}\frac{d}{dX}T_{j-1}(X)=\sum_k k a_{j-1,k} X^{k-1}=\sum_k (k+1) a_{j-1,k+1} X^k.\end{equation} Now we deduce from (\ref{rtpjtpjmutjmu}),(\ref{coeftjmo}),(\ref{dcoeftj}) and (\ref{dcoeftjmo}) that $$(k+1) a_{j,k+1}=j((k+1) a_{j-1,k+1}+a_{j-1,k}),$$ and reindexing the $k+1$ we get the desired result.
\end{proof}

\begin{lemma} $$a_{j,1}=\frac{j(j+1)!}{2}\ \forall j\in\Bbb Z^+.$$
\end{lemma}

\begin{proof} We will prove it by induction on $j$. For $j=0$, it holds trivially. Let us suppose that it is true for $j-1$. By the previous lemma, $$a_{j,1}=j(a_{j-1,1}+a_{j-1,0})$$ and since by Proposition \ref{cpolpi2} $$a_{j-1,0}=j!$$ and by the inductive hypothesis $$a_{j-1,1}=\frac{(j-1)j!}{2},$$ we obtain that $$a_{j,1}=j(\frac{(j-1)j!}{2}+j!)=j\cdot j!(\frac{j+1}{2})=\frac{j(j+1)!}{2}.$$
\end{proof}

\begin{lemma} $$k(k+1)a_{j,k}=(j+1-k)a_{j,k-1}\ \forall j\in\Bbb Z^+,\forall k\in\Bbb Z$$
\end{lemma}

\begin{proof} If $k\leq 0$ the equality holds obviously. We will prove the result by induction on $j$. For $j=0$ it is trivially true.
Let us suppose that it is true for $j-1$. For $k=1$ the result follows from the previous lemma. For $k>1$, by multiplying by $k-1$ we obtain that what we want to prove is equivalent to $$(k^2-1)k a_{j,k}=(j+1-k)(k-1) a_{j,k-1},$$ and by Lemma \ref{recurrcoefs}, this is the same as  \begin{equation}\label{bsj}(k^2-1)j(k a_{j-1,k}+a_{j-1,k-1})=(j+1-k)j((k-1)a_{j-1,k-1}+a_{j-1,k-2}).\end{equation} After symplifying $j$ and regrouping terms we can see that (\ref{bsj}) is equivalent to \begin{equation}\label{recurrcoefs2}(k^2-1)k a_{j-1,k}+(2k-j)(k-1)a_{j-1,k-1}=(j+1-k)a_{j-1,k-2}.\end{equation} By the induction hypothesis, $$k(k+1)a_{j-1,k}=(j-k)a_{j-1,k-1},$$ and hence $$(k^2-1)k a_{j-1,k}=(k-1)(j-k)a_{j-1,k-1},$$ and therefore (\ref{recurrcoefs2}) is equivalent to $$(k-1)k a_{j-1,k-1}=(j+1-k)a_{j-1,k-2},$$ and this is true by induction hypothesis.
\end{proof}

\begin{theorem} $$a_{j,k}={j \choose k} {j+1 \choose k+1}(j-k)!\ \forall j,k\in\Bbb Z^+.$$
\end{theorem}

\begin{proof} It is easily proved by induction on $k$ using the previous lemma and the fact that, by Proposition \ref{cpolpi2}, $$a_{j,0}=(j+1)!$$
\end{proof}

In next corollary $L^{(1)}_j (X)$ is the $j$-th generalized Laguerre polynomial corresponding to $\alpha=1$ in the family of generalized Laguerre polynomials $L^{(\alpha)}_j(X)$.

\begin{corollary} $$T_j(-X)=j!L^{(1)}_j (X)\ \forall j\in\Bbb Z^+.$$
\end{corollary}

\begin{proof}$$T_j(-X)=\sum_{k=0}^j\frac{j!}{k!(j-k)!}{j+1 \choose k+1}(j-k)!(-1)^k X^k=j!\sum_{k=0}^j {j+1 \choose k+1}\frac{(-1)^k}{k!} X^k=j!L^{(1)}_j (X).$$
\end{proof}

\begin{corollary} The family of polynomials $T_j(X)$ satisfy the differential equation $$X\frac{d^2 Y}{dX^2}+(2+X)\frac{d Y}{dX}-jY=0.$$
\end{corollary}

\begin{proof} Generalized Laguerre polynomials $L^{(1)}_j (Z)$ satisfy the differential equation $$Z\frac{d^2 Y}{d Z^2}+(2-Z)\frac{d Y}{d Z}+jY=0,$$ and obviously polynomials proportional to them also satisfy it. Now we make the change of variable $X=-Z$.
\end{proof}

We will give, in a similar way as we did with directed paths, the tabloidal form of $L^{(1)}_j(X)$:

\begin{theorem}$$L^{(1)}_j(X)=(\sum_{i=0}^{\infty}{i+j+1 \choose j}\frac{1}{i!}(-X)^i)\exp(X)\forall j\in\Bbb Z^+.$$
\end{theorem}

\begin{proof} It follows the line of the proof of Theorem \ref{tfp}, having into account that in this case, by Theorem \ref{contcc}, $$\sigma(E_{j+1}+P_{i+1})={i+j+1 \choose i+1,1,\dots,1}=\frac{(i+j+1)!}{(i+1)!}.$$
\end{proof}

At this point we need a systematic way to associate a differential equation to a digraph with a distinguished vertex. We will do this through the use of the companion polynomial. First we will establish a way to obtain a Laguerre-like differential equation from a polynomial.

\begin{theorem}\label{lcxidilag} If $P\in\Bbb Q[X]$ and $\deg(P)=n$, then $P$ can be expressed in a unique way as a linear combination of $$L_n(X),X\frac{d}{dX}L_n(X),\dots,X^n\frac{d^n}{dX^n}L_n(X).$$
\end{theorem}

\begin{proof} Since $$X^i\frac{d^i}{d X^i} L_n(X)=\sum_{j=0}^{n-i} {n \choose i+j} \frac{(-1)^{i+j}}{j!} X^{i+j}\ \forall i\in\Bbb Z^+,$$ the matrix formed by the coefficients of these polynomials is triangular, and its determinant is $$(-1)^{\lfloor (n+1)/2\rfloor} {n \choose 0}{n \choose 1}\cdot\dots\cdot {n \choose n},$$ where $\lfloor\rfloor$ denotes the integer part, and therefore it is non-zero.
\end{proof}

Observe that if $n,i\in\Bbb Z^+$, we can put \begin{equation}\label{dxilnx}\frac{d^i}{dX^i}L_n(X)=A_{n,i}(X)L_n(X)+B_{n,i}(X)\frac{d}{dX}L_n(X),\end{equation} where $A_{n,i}(X),B_{n,i}(X)$ are rational fractions (with denominator $X^{n-1}$ when $n\geq 1$). More specifically, we have:

\begin{proposition}\label{stairab} The coefficients $A_{n,i}(X),B_{n,i}(X)$ satisfy the following recurrence relations:
$$A_{n,0}(X)=1,B_{n,0}(X)=0$$ and, if $i\geq 1$, then $$A_{n,i}(X)=\frac{d}{dX} A_{n,i-1}(X)-\frac{n}{X}B_{n,i-1}(X),$$ $$B_{n,i}(X)=A_{n,i-1}(X)+\frac{d}{dX} B_{n,i-1}(X)+\frac{X-1}{X}B_{n,i-1}(X).$$
\end{proposition}

\begin{proof} Use the fact that \begin{equation}\label{sdlp} \frac{d^2}{dX^2}L_n(X)=\frac{-n}{X}L_n(X)+\frac{X-1}{X}\frac{d}{dX}L_n(X).\end{equation}
\end{proof}

\begin{theorem} If $P[X]\in\Bbb Q[X]$ and $\deg(P[X])=n$, then $P$ can be expressed as $$P(X)=Q(X)L_n(X)+R(X)\frac{d}{dX}L_n(X),$$ where $Q(X),R(X)$ are polynomials in $\Bbb Q[X]$ of degrees at most $n-1$ and $n$, respectively.
\end{theorem}

\begin{proof} It is an immediate consequence of Theorem \ref{lcxidilag} and the comment before the previous proposition.
\end{proof}

In the next theorem we will use prime notation for first and second derivatives.

\begin{theorem}\label{lde} If $Y=P(X)$ is a polynomial of degree $n$ in $\Bbb Q[X]$ and $Q(X),R(X)$ are as in the previous theorem, then $X Y'' (R(X) (n R(X)-X Q'(X))+Q(X) (X R'(X)+(X-1) R(X))+X Q(X)^2)+Y' (R(X) (X (X Q''(X)-2 n R'(X))-n (X-2) R(X))-Q(X)(X (2 X Q'(X)+X R''(X)+2 (X-1) R'(X))+(X^2-2 X+2) R(X))+(1-X) X Q(X)^2)+Y (Q(X) (X (3 n R'(X)-X Q''(X)+(X-1) Q'(X))+n(X-2) R(X))+R(X) (X ((1-X) Q''(X)-n R''(X))+(-(3 n+2) X+X^2+2) Q'(X)+n (X-2) R'(X))+X (R'(X) (2 n R'(X)-X Q''(X))+Q'(X)(X R''(X)+2 (X-1) R'(X))+2 X Q'(X)^2)+n X Q(X)^2+(n-1) n R(X)^2)$
\end{theorem}

\begin{proof} Derivate twice in the expression for $P(X)$ obtained in the previous theorem, using (\ref{sdlp}), and eliminate $L_n(X)$ and $\frac{d}{dX}L_n(X)$ after that.
\end{proof} 

We will call the previous differential equation the laguerrean differential equation associated to $P(X)$, and we will denote it by $\frak L (P(X))$. It is of the form $$U(X)Y^{\prime\prime}+V(X)Y^{\prime}+W(X)=0,$$ with $U(X),V(X),W(X)$ polynomials in $\Bbb Q[X]$. Of course, when $P(X)=L_n(X)$, the laguerrean is just Laguerre's equation. The laguerrean just defined is very general, because $P(X)$ is an arbitrary polynomial, but the problem is that the degree of $U(X),V(X),W(X)$ is very high. Fortunately, for families of companion polynomials of digraphs evaluated in $-X$ we will find families of differential equations in which the degrees of $U(X),V(X),W(X)$ are low, and therefore are more close to Laguerre polynomials and Laguerre equation. Besides, the families of polynomials that we will find are sometimes orthogonal.

\section{Differential equation associated to a dispositional digraph with two lines}\label{disptwolines}

For $n_1,n_2\in\Bbb Z^+$ with $n_1\leq n_2$, let $\frak G_{n_1,n_2}$ be the digraph shown in the following figure:

\begin{figure}[H]
\caption{Digraph $\frak G_{n_1,n_2}$}
\includegraphics[width=5 in]{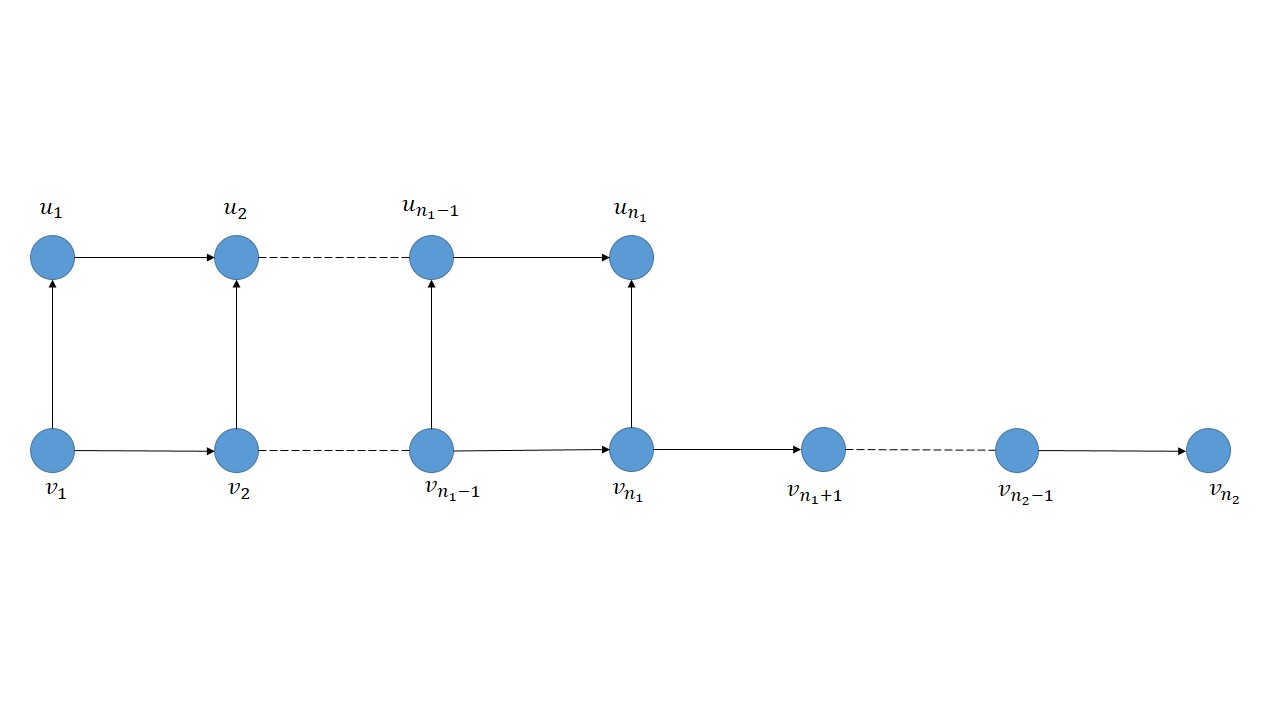}
\end{figure}

with $n_1$ vertices $u_1,\dots,u_{n_1}$ in the upper row and $n_2$ vertices $v_1,\dots,v_{n_2}$ in the lower row.

It corresponds to the dispositional digraph $\mathcal G([n_2,0],[n_1,0])$.

If $r\in\{1,\dots,n_2\}$, we can take the vertex $v=v_r$, and we consider for $i\in\Bbb Z^+$ the corresponding digraphs $(\frak G_{n_1,n_2})_i$ obtained by plugging a directed path $P_{i+1}$ at vertex $v$ as we described in Section \ref{cpdigraph}. We show in the following figure the case when $r=2$,

\begin{figure}[H]
\caption{Digraphs $(\frak G_{n_1,n_2})_i$ for $v=v_2$}
\includegraphics[width=5 in]{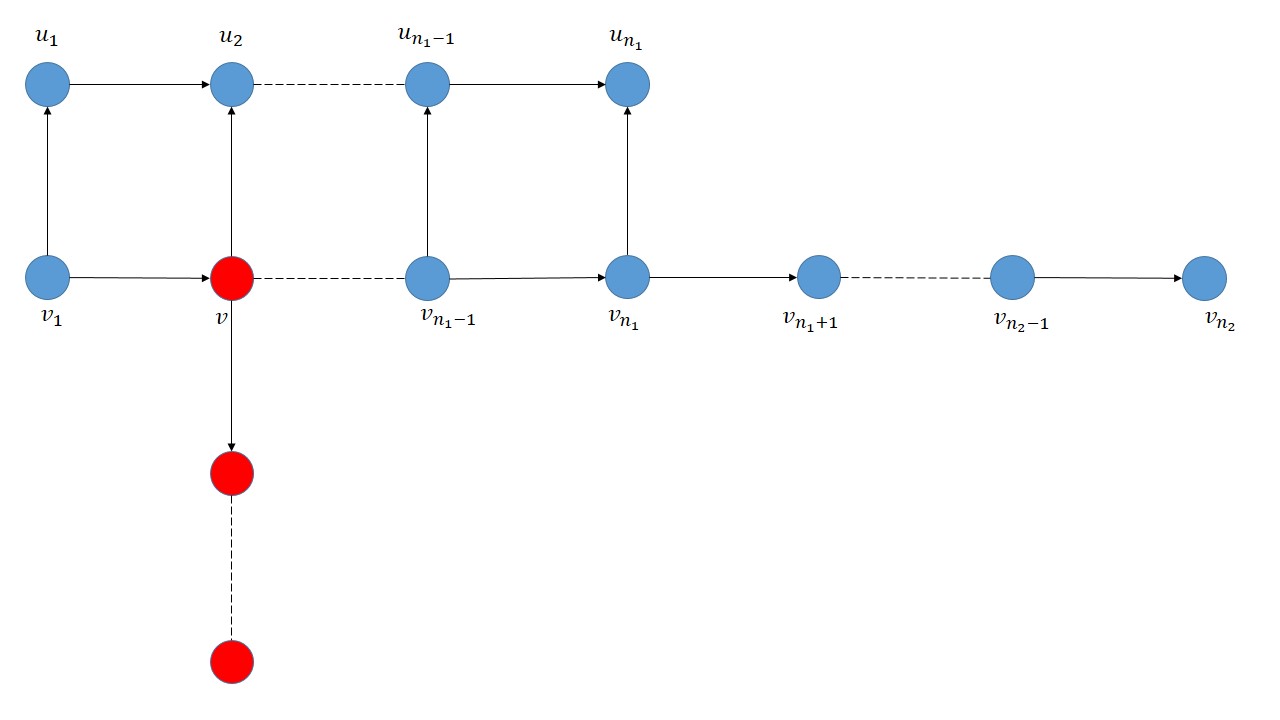}
\end{figure}

\begin{theorem}\label{recrelcp} If $0\leq n_1\leq n_2$ and $n_2\geq r$, then $$T_{\frak G_{n_1,n_2},v}(X)=\sigma(\frak G_{n_1,n_2})\sum_{i\in\Bbb Z^+} f(n_1,n_2,r,i) X^i\ \frac{d^i L_{n_1+n_2-r}}{d X^i}(-X),$$ where the $f(n_1,n_2,r,i)$ satisfy the following recurrence relations:

\begin{enumerate}
\item $f(n_1,n_2,r,i)=0$ for $i>\min\{n_1,r-1\}$
\item $f(n_1,n_2,r,i)=\frac{n_1(n_1-n_2-2)(n_1+n_2-r-i)}{(n_1-n_2-1)(n_1+n_2)(n_1+n_2-r)} f(n_1-1,n_2,r,i)+\frac{(n_1-n_2)(n_2+1)(n_1+n_2-r-i)}{(n_1-n_2-1)(n_1+n_2)(n_1+n_2-r)} f(n_1,n_2-1,r,i)$, for $1\leq n_1\leq n_2-1$, $n_2\geq r+1$, $i\leq\min\{n_1,r-1\}$\newline

\item $f(n_1,r,r,i)=\frac{(i-n_1)(n_1-r-2)}{r^2+r+n_1(1-n_1)}f(n_1-1,r,r,i)+\frac{r^2+(1-n_1)r-n_1}{i!(r^2+r+n_1(1-n_1))}$, for $1\leq n_1\leq r,i\leq n_1$

\item $f(0,n_2,r,0)=1$, for $n_2\geq r$\newline

\item $f(n_1,n_1,r,i)=\frac{2n_1-r-i}{2n_1-r}f(n_1-1,n_1,r,i)$, for $n_1>r$, $i\leq r-1$.
\end{enumerate}
\end{theorem}

\begin{proof} Let $$U(n_1,n_2,r,X)=\sigma(\frak G_{n_1,n_2})\sum_{i\in\Bbb Z^+} f(n_1,n_2,r,i) X^i\ \frac{d^i L_{n_1+n_2-r}}{d X^i}(-X),$$ where the values $f(n_1,n_2,r,i)$ are determined by the recurrence relations described in the statement of this theorem.

 First, we will prove by induction on $n_1+n_2$ that $f(n_1,n_2,r,0)=1\ \forall n_1,n_2,r$. It is trivially true for $n_1+n_2=1$. Let us assume, by induction hypothesis, that it holds for $n^{\prime}_1+n^{\prime}_2=n_1+n_2-1$. Then it holds also for $n^{\prime}_1+n^{\prime}_2=n_1+n_2$, because under the conditions indicated in the second recurrence relation in the statement of this theorem we have $$\frac{n_1(n_1-n_2-2)(n_1+n_2-r)}{(n_1-n_2-1)(n_1+n_2)(n_1+n_2-r)}+\frac{(n_1-n_2)(n_2+1)(n_1+n_2-r)}{(n_1-n_2-1)(n_1+n_2)(n_1+n_2-r)}=1,$$ under the conditions of the third one, $$\frac{-n_1(n_1-r-2)}{r^2+r+n_1(1-n_1)}+\frac{r^2+(1-n_1)r-n_1}{(r^2+r+n_1(1-n_1))}=1,$$ under the conditions of the fourth one we get also $1$, and under the conditions of the fifth one, $$\frac{2n_1-r}{2n_1-r}=1.$$

 Next, we will prove that $U(n_1,n_2,r,X)$ and $\frac{\partial U(n_1,n_2,r,X)}{\partial X}$ satisfy the recurrence obtained in Theorem \ref{rcp}. We need to distinguish the possible casses that can occur. Since the proof is very similar in all cases, we will show it only for $r+1\leq n_1\leq n_2-1$ and $0\leq i\leq r-1$. In this case, $\min (\frak G_{n_1,n_2})-\{v\}=\{u_{n_1},v_{n_2}\}$.

If $n\in\Bbb Z^+$ and $0\leq i\leq n$, then $$\frac{d^i L_n}{d X^i}(-X)=\sum_{j=0}^n \frac{n!(-1)^i}{j!(n-j)!(j-i)!} X^{j-i}.$$ (Obviously, by convention we will agree that a quotient in which the factorial of a negative number appears in the denominator is $0$, so that the index $j$ in the summation goes from $i$ to $n$). We deduce from the previous identity, using the obvious formula \begin{equation}\label{sgnund}\sigma(\frak G_{n_1,n_2})={n_1+n_2 \choose n_1}-{n_1+n_2 \choose n_1-1},\end{equation} that \begin{equation}\label{upnundrx}\begin{split}\frac{\partial U(n_1,n_2,r,X)}{\partial X}-U(n_1-1,n_2,r,X)-U(n_1,n_2-1,r,X)-\frac{\partial U(n_1-1,n_2,r,X)}{\partial X}-\\
\frac{\partial U(n_1,n_2-1,r,X)}{\partial X}=\sum_{j=0}^{n_1+n_2-r-1}\sum_{i=0}^{r-1} \frac{(-1)^i(n_1+n_2-r-1)!}{j!(j-i+1)!(n_1+n_2-r-j-1)!}\\
((i-n_1-n_2+r)(\sigma(\frak G_{n_1-1,n_2})f(n_1-1,n_2,r,i)+\sigma(\frak G_{n_1,n_2-1})f(n_1,n_2-1,r,i))-\\
(n_1+n_2-r)\sigma(\frak G_{n_1,n_2})f(n_1,n_2,r,i))X^j\end{split}\end{equation} Now we obtain from (\ref{sgnund}) and the second recurrence relation in the statement of this theorem that all the coefficients of the powers of $X$ in the right-hand side in (\ref{upnundrx}) are zero, and therefore that $$\frac{\partial U(n_1,n_2,r,X)}{\partial X}=U(n_1-1,n_2,r,X)+U(n_1,n_2-1,r,X)+\frac{\partial U(n_1-1,n_2,r,X)}{\partial X}+\frac{\partial U(n_1,n_2-1,r,X)}{\partial X},$$ as desired. This proves that for every $n_1,n_2,r$ there exists $c_{n_1,n_2,r}\in\Bbb Q$ such that $U(n_1,n_2,r,X)=T_{\frak G_{n_1,n_2},v}(X)+c_{n_1,n_2,r}$. Substituting $X$ by $0$, and having into account that by Theorem \ref{rcp} we have that $T_{\frak G_{n_1,n_2},v}(0)=\sigma(\frak G_{n_1,n_2})$ and that $X^i\ \frac{d^i L_{n_1+n_2-r}}{d X^i}(0)=0\ \forall i>0$, we get $$\sigma(\frak G_{n_1,n_2})f(n_1,n_2,r,0)L_{n_1+n_2-r}(0)=\sigma(\frak G_{n_1,n_2})+c_{n_1,n_2,r}$$ and, since $f(n_1,n_2,r,0)=1$ and $L_{n_1+n_2-r}(0)=1$, we obtain that $c_{n_1,n_2,r}=0\ \forall n_1,n_2,r$, and we are done.
\end{proof}

We will study the associated differential equations for the cases when $r=2,3$:

\begin{proposition}\label{chin2} For $r=2$ it holds that $$T_{\frak G_{n_1,n_2},v}(X)=\begin{cases} 1 & \text{if }n_1=0,n_2=2,\\ ({n_1+n_2 \choose n_1}-{n_1+n_2 \choose n_1-1})(L_{n_1+n_2-2}(-X)+\frac{(n_1n_2+n_1)X\frac{d L_{n_1+n_2-2}}{dX}(-X)}{(n_1+n_2)(n_1+n_2-1)(n_1+n_2-2)}) & \text{elsewhere.}\end{cases}$$
\end{proposition}

\begin{proof} It is easy to check by doing elementary algebraic operations that the values satisfy the recurrence relations of Theorem \ref{recrelcp}.
\end{proof}

Hereinafter, we will use the Pochhammer symbol $(z)_n$ to denote the product $$z(z+1)(z+2)\cdot\dots\cdot (z+n-1).$$

We will use the previous proposition to obtain differential equations satisfied by $T_{\frak G_{n_1,n_2},v}(X)$. 

\begin{theorem}\label{diffeqn1n2} For $r=2$, let $$\alpha(n_1,n_2)=(n_1+n_2-2)_3 (n_1+n_2-1)_2,$$ $$\beta(n_1,n_2)=n_1(n_2+1)(n_1^2+n_2^2+n_1n_2-2n_1-n_2),$$ and $$\gamma(n_1,n_2)=(n_1+n_2-2)_3 (n_1^3+n_2^3+3n_1^2n_2+3n_1n_2^2-3n_1^2-3n_2^2-7n_1n_2+n_1+2n_2).$$ Then the polynomials $T_{\frak G_{n_1,n_2},v}(X)$ satisfy the following differential equation for every $n_1,n_2$ with $0\leq n_1\leq n_2$ and $n_2\geq 2$: $$(\alpha X+\beta X^2) Y^{\prime\prime}+(\alpha+\alpha X+\beta X^2) Y^{\prime}-(\gamma+(n_1+n_2-2)\beta X) Y=0.$$ 
\end{theorem}

\begin{proof} Use Proposition \ref{chin2} and Theorem \ref{lde}.
\end{proof}

For instance, for $n_1=2,n_2=3$ we have that $$T_{\frak G_{2,3},v}(X)=\frac{1}{2}X^3+\frac{11}{2}X^2+13X+5,$$ and the associated differential equation, after factoring out $48$, is $$(25X+2X^2)Y^{\prime\prime}+(25+25X+2X^2)Y^{\prime}-(65+6X)Y=0$$

The differential equations in the previous theorem generalize Laguerre equation, in the sense that, if we do in $Y(X)$ the change of variable $Y=-Z$ then the obtained equation when $n_1=0$ is (up to a constant factor which does not change essentially the equation) Laguerre equation.

When $n_1=n_2$ we have that $\sigma(\frak G_{n_1,n_2})={2n_1 \choose n_1}-{2n_1 \choose n_1-1}=C_{n_1}$, the $n_1$-th Catalan number. In this case the differential equation is simpler than in the general case:

\begin{theorem} If $n\geq 2$, then the polynomials $T_{\frak G_{n,n},v}(X)$ satisfy the differential equation

$(3(n+1)X^2+8(2n-1)^2 X)Y^{\prime\prime}+$

$(3(n+1)X^2+8(2n-1)^2 X+8(2n-1)^2)Y^{\prime}+$

$(6(n+1)(-n+1) X-4(2n-1)(8n^2-13n+3))Y=0.$
\end{theorem}

\begin{proof} Use Theorem \ref{diffeqn1n2} and simplify the factor $n^2(n-1)$.
\end{proof}

So, we will call the differential equation of the previous theorem the {\em Catalan differential equation}, and the associated polynomials $C^{\star}_n(X)=T_{\frak G_{n,n},v}(X)$ the {\em Catalan polynomials} (Do not confuse with the concept of $(q,t)$-Catalan polynomials $C_n(q,t)$ (\cite{qtcp}), which generalize classical Catalan numbers).

Catalan polynomials for $i$ from $2$ to $6$ and their graphical representations, as well as the associated Catalan differential equations, are given next:

For $i=2$,

$$C^{\star}_2(X)=\frac{X^2}{2}+3 X+2$$

\begin{figure}[H]
\caption{Graph of $C^{\star}_2(X)$}
\includegraphics[width=2.5in]{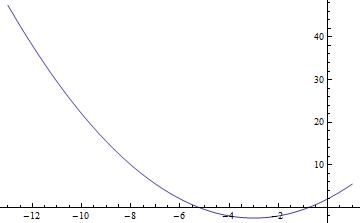}
\end{figure}

Catalan differential equation associated to $C^{\star}_2(X)$:

$$\left( 9\,{X}^{2}+72\,X \right) {\it Y^{\prime\prime}}+ \left( 9\,{X}^{2}+72\,X+72
 \right) {\it Y^{\prime}}- \left( 108+18\,X \right) Y=0$$

For $i=3$,

$$C^{\star}_3(X)= \frac{X^4}{8}+\frac{7 X^3}{3}+12 X^2+18 X+5$$

\begin{figure}[H]
\caption{Graph of $C^{\star}_3(X)$}
\includegraphics[width=2.5in]{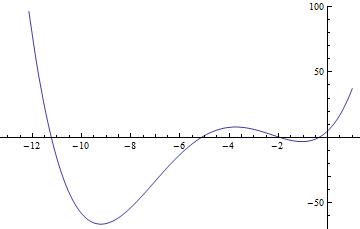}
\end{figure}

Catalan differential equation associated to $C^{\star}_3(X)$:

$$\left( 12\,{X}^{2}+200\,X \right) {\it Y^{\prime\prime}}+ \left( 12\,{X}^{2}+200\,X
+200 \right) {\it Y^{\prime}}- \left( 720+48\,X \right) Y=0
$$

For $i=4$,

$$C^{\star}_4(X)=\frac{X^6}{80}+\frac{59 X^5}{120}+\frac{20 X^4}{3}+\frac{115 X^3}{3}+\frac{185 X^2}{2}+79 X+14$$

\begin{figure}[H]
\caption{Graph of $C^{\star}_4(X)$}
\includegraphics[width=2.5in]{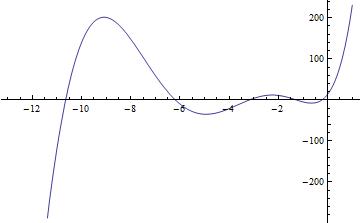}
\end{figure}

Catalan differential equation associated to $C^{\star}_4(X)$:

$$\left( 15\,{X}^{2}+392\,X \right) {\it Y^{\prime\prime}}+ \left( 15\,{X}^{2}+392\,X
+392 \right) {\it Y^{\prime}}- \left( 2212+90\,X \right) Y=0$$

For $i=5$,

$$C^{\star}_5(X)=\frac{X^8}{1440}+\frac{17 X^7}{360}+\frac{49 X^6}{40}+\frac{931 X^5}{60}+\frac{1225 X^4}{12}+343 X^3+539 X^2+322 X+42$$

\begin{figure}[H]
\caption{Graph of $C^{\star}_5(X)$}
\includegraphics[width=2.5in]{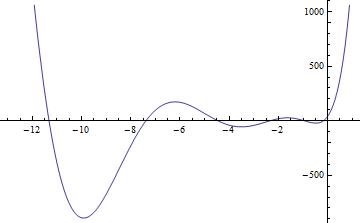}
\end{figure}

Catalan differential equation associated to $C^{\star}_5(X)$:

$$\left( 18\,{X}^{2}+648\,X \right) {\it Y^{\prime\prime}}+ \left( 18\,{X}^{2}+648\,X
+648 \right) {\it Y^{\prime}}- \left( 4968+144\,X \right) Y=0$$

For $i=6$,

$$C^{\star}_6(X)=\frac{X^{10}}{40320}+\frac{157 X^9}{60480}+\frac{123 X^8}{1120}+\frac{171 X^7}{70}+\frac{623 X^6}{20}+\frac{2331 X^5}{10}+1008 X^4+2388 X^3+2781 X^2+1278 X+132$$

\begin{figure}[H]
\caption{Graph of $C^{\star}_6(X)$}
\includegraphics[width=2.5in]{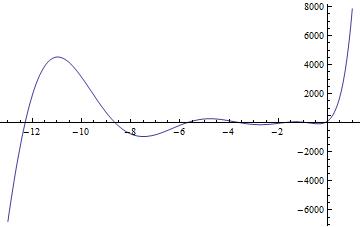}
\end{figure}

Catalan differential equation associated to $C^{\star}_6(X)$:

$$\left( 21\,{X}^{2}+968\,X \right) {\it Y^{\prime\prime}}+ \left( 21\,{X}^{2}+968\,X
+968 \right) {\it Y^{\prime}}- \left( 9372+210\,X \right) Y=0$$

We will prove next that Catalan polynomials are orthogonal, changing the sign of the independent variable, with respect to the inner product $<f,g>=\int_0^{+\infty}f(X)g(X)\exp(-X)\ dX$.

\begin{theorem} For any $i\not=j$, $$\int_0^{+\infty}C^{\star}_i(-X)C^{\star}_j(-X)\exp(-X)\ dX=0.$$
\end{theorem}

We will use several lemmas to prove the theorem:

\begin{lemma}\label{intxidlexp} For all $i,j,n\in\Bbb Z^+$, $$\int_0^{+\infty}X^i\frac{d^j}{dX^j}L_n(X)\exp(-X)\ dX=\begin{cases} (-1)^n (n-j+1)_j (i-n+1)_{n-j} (n+1)_{i-n}, & \text{ if }j\leq n\\ 0, & \text{ otherwise} \end{cases}$$
\end{lemma}

\begin{proof} It can be done by induction on $n$, using the recurrence relation $$L_{n+1}^{(\alpha)}(X)=\frac{(2n+1+\alpha-X)L_n^{(\alpha)}(X)-(n+\alpha)L_{n-1}^{(\alpha)}(X)}{n+1}$$ for the generalized Laguerre polynomials $L_n^{(\alpha)}(X)$, and having into account that $$\frac{d^j}{d X^j}L_n(X)=\begin{cases} (-1)^j L_{n-j}^{(j)}(X), & \text{ if }j\leq n\\ 0, & \text{otherwise} \end{cases}$$
\end{proof}

\begin{lemma}\label{intiszero} $$\int_0^{+\infty}X^i\frac{d^j}{dX^j}L_n(X)\exp(-X)\ dX=0$$ whenever $j\leq i<n$.
\end{lemma}

\begin{proof} It is a consequence of the previous lemma, because in this case $(i-n+1)_{n-j}=0$.
\end{proof}

\begin{lemma}\label{lop} $$\int_0^{+\infty}X^{j+k}\frac{d^j}{dX^j}L_n(X)\frac{d^k}{dX^k}L_m(X)\exp(-X)\ dX=0$$ whenever $\vert m-n\vert>\max\{j,k\}$
\end{lemma}

\begin{proof} It can be easily deduced from the previous lemma.
\end{proof}

Now the proof of the theorem follows directly from Lemma \ref{lop} and Proposition \ref{chin2}.

Although $T_{\frak G_{n_1,n_2},v}(X)$ is defined only for $n_2\geq 2$, if we put $C^{\star}_1(X)=1$ (which is consistent with the form of the $T_{\frak G_{n_1,n_2},v}(X)$ in Theorem \ref{recrelcp} and with the fact that $f(n_1,n_2,r,0)=1$ stated at the beginning of its proof), then it is easy to prove that, in fact, the family $$B=\{C^{\star}_1(-X),C^{\star}_2(-X),C^{\star}_3(-X),\dots\}$$ is orthogonal.

Next, we will prove the inextendibility of the family of Catalan polynomials. Since the polynomials $C^{\star}_1(-X),C^{\star}_2(-X),C^{\star}_3(-X),\dots$ are orthogonal, the family $B$ is free. We will prove that it is maximal when subjected to the orthogonality condition:

\begin{theorem} The family $B$ is a maximal orthogonal family.
\end{theorem}

\begin{proof} Let us supposse that $B\subsetneq B^{\prime}$, where $B^{\prime}$ is orthogonal, and let $P(X)\in B^{\prime}-B$. Then, $P(X)$ is of odd degree, because $B$ generates the set of polynomials of even degree. But, if $P(X)$ is of degree $2n+1$ and $P(X)=\displaystyle\sum_{i=0}^{2n+1} a_i X^i$ then we have, by Lemmas \ref{intxidlexp} and \ref{intiszero}, and by Proposition \ref{chin2}, that $$\int_0^{+\infty} P(X) C^{\star}_{n+2}(-X)\exp(-X)\ dX=\frac{-a_{2n+1}C_{n+2}(n+3)(2n+1)!}{2(2n+3)}\not=0,$$ and we get a contradiction.
\end{proof}

Thus, despite the fact that in linear algebra a free family can always be extended to a base, we see that when we take orthogonality into account the situation is different if we want the extension to remain orthogonal.

Now we will analyze the case $r=3$. The proofs are similar to the ones for $r=2$ and will be omitted.

\begin{proposition} For $r=3$ it holds that $$T_{\frak G_{n_1,n_2},v}(X)=\begin{cases} 1 & \text{if }n_1=0,n_2=3,\\ X+1 & \text{if }n_1=0,n_2=4,\\ X+3 & \text{if }n_1=1,n_2=3,\\ \begin{split}({n_1+n_2 \choose n_1}-{n_1+n_2 \choose n_1-1})(L_{n_1+n_2-3}(-X)+\\
\frac{2 n_1(n_2+1)(n_1^2+3n_1n_2+n_2^2-5 n_1-6 n_2+6)XL^{\prime}_{n_1+n_2-3}(-X)}{(n_1+n_2-3)_4(n_1+n_2-3)}+\\
\frac{2 (n_1-1)_2(n_2)_2 X^2L^{\prime\prime}_{n_1+n_2-3}(-X)}{(n_1+n_2-4)_5(n_1+n_2-3)})\end{split} & \text{elsewhere.}\end{cases}$$
\end{proposition}

\begin{theorem} For $r=3$, let

$\alpha(n_1,n_2)=(n_1+n_2-4)_5 (n_1+n_2-3)_4$,\newline

$\beta(n_1,n_2)=2 n_1 (n_2+1) (n_1^6+7 n_1^5 (n_2-2)+n_1^4 (15 n_2^2-72 n_2+77)+n_1^3 (16 n_2^3-119 n_2^2+279 n_2-208)+n_1^2 (15 n_2^4-107 n_2^3+333 n_2^2-490 n_2+276)+n_1 (7 n_2^5-67 n_2^4+227 n_2^3-384 n_2^2+352 n_2-144)+n_2 (n_2^5-15 n_2^4+74 n_2^3-156 n_2^2+144 n_2-48))$,\newline

$\gamma(n_1,n_2)=2 (-1 + n_1)_2 (n_2)_2 (-2 + n_1 + n_2) (-3 n_1 + n_1^2 + (-1 + n_2) n_2)$,\newline

$\delta(n_1,n_2)=2 n_1 (n_2+1) (n_1^6+7 n_1^5 (n_2-2)+n_1^4 (15 n_2^2-73 n_2+77)+n_1^3 (16 n_2^3-120 n_2^2+285 n_2-208)+n_1^2 (15 n_2^4-108 n_2^3+338 n_2^2-501 n_2+276)+n_1 (7 n_2^5-68 n_2^4+231 n_2^3-390 n_2^2+358 n_2-144)+n_2 (n_2^5-15 n_2^4+75 n_2^3-159 n_2^2+146 n_2-48))$,\newline

$\epsilon=(-4 + n_1 + n_2)_5 (6 n_1 - 29 n_1^2 + 27 n_1^3 - 9 n_1^4 + n_1^5 + 18 n_2 - 78 n_1 n_2 + 91 n_1^2 n_2 - 38 n_1^3 n_2 + 5 n_1^4 n_2 - 39 n_2^2 + 97 n_1 n_2^2 - 60 n_1^2 n_2^2 + 10 n_1^3 n_2^2 + 29 n_2^3 - 38 n_1 n_2^3 + 10 n_1^2 n_2^3 - 9 n_2^4 + 5 n_1 n_2^4 + n_2^5)$,\newline

$\zeta=2 n_1 (n_2+1) (n_1+n_2-4) (n_1^6+n_1^5 (7 n_2-13)+n_1^4 (15 n_2^2-68 n_2+67)+n_1^3 (16 n_2^3-114 n_2^2+249 n_2-171)+n_1^2 (15 n_2^4-102 n_2^3+302 n_2^2-414 n_2+216)+n_1 (7 n_2^5-63 n_2^4+203 n_2^3-327 n_2^2+282 n_2-108)+n_2 (n_2^5-14 n_2^4+65 n_2^3-130 n_2^2+114  n_2-36)).$\newline

Then the polynomials $T_{\frak G_{n_1,n_2},v}(X)$ satisfy the following differential equation for every $n_1,n_2$ with $0\leq n_1\leq n_2$ and $n_2\geq 3$:
$$(\alpha X+\beta X^2+\gamma X^3) Y^{\prime\prime}+(\alpha+\alpha X+\delta X^2+\gamma X^3) Y^{\prime}-(\epsilon+\zeta X+(n_1+n_2-3)\gamma X^2) Y=0$$
\end{theorem}

Just as it was the case for $r=2$, when $n_1=n_2=n$ the differential equation takes a simpler form. In this case the factor $4(n-2)_3(n-1)_2$ can be cancelled, and we get\newline

$((72 - 384 n + 704 n^2 - 512 n^3 + 128 n^4) X + (48 - 25 n - 42 n^2 + 31 n^3) X^2 + (2 n + 2 n^2) X^3) Y^{\prime\prime}+(72 - 384 n + 704 n^2 - 512 n^3 + 128 n^4 + (72 - 384 n + 704 n^2 - 512 n^3 + 128 n^4) X + (48 - 27 n - 44 n^2 + 31 n^3) X^2 + (2 n + 2 n^2) X^3) Y^{\prime}-2( (-72 + 558 n - 1438 n^2 + 1560 n^3 - 744 n^4 + 128 n^5) + (-72 + 90 n + 37 n^2 - 94 n^3 + 31 n^4) X +  (-3 n - n^2 + 2 n^3) X^2) Y=0$\newline

We will call the diferential equation of the previous theorem the {\em Catalan differential equation} of order $3$, and the associated polynomials $C^{\star 3}_n=T_{\frak G_{n,n},v}$ the {\em Catalan polynomials} of order $3$.

The property of orthogonality of Catalan polynomials which held for $r=2$ is no longer true for $r=3$; for instance, $$\int_0^{+\infty}C^{\star 3}_3(-X)C^{\star 3}_4(-X)\exp(-X)\ dX=4$$

\section{Staircase differential equations}\label{sdec}

We will consider, for $n\in\Bbb N$, the staircase digraph $\frak S_n$, the vertex $v=v_1$, and the companion polynomial $T_{\frak S_n,v}(X)$. In this section we will do a numerical study of the first polynomials and their associated differential equations.

\begin{proposition} If we call $f(n,i)$ to $\sigma(\frak S_n+P_{i+1})$, then $f(n,i)$ satisfies the recurrence relation

$$f(n,i)=\frac{f(n,i-1)+\sum_{j=0}^{n-1} {n+i-1 \choose i+j} f(j,i)s_{n-1-j}}{2}\text{ whenever }n\geq 2,i\geq 1,$$

with the initial conditions

$$f(0,i)=1$$

$$f(1,i)=1$$

$$f(n,0)=s_n$$

where the $s_{n-1-j}$ and the $s_n$ are the corresponding Euler zigzag numbers.

\end{proposition}

\begin{proof} Use Theorem \ref{rec} for both the minimum points and the maximum points, and sum the obtained identities.
\end{proof}

Note that the initial condition $f(1,i)=1$ in the previous proposition might seem redundant at first glance, but it really is not, and the result is false if we remove it.

As a consequence of the previous proposition, the following result can be easily obtained inductively:

\begin{proposition}\label{stairfni} $$f(n,i)=(-1)^{\lfloor\frac{i}{2}\rfloor}s_{n+i}+\sum_{j=0}^{\lfloor\frac{i}{2}\rfloor-1}(-1)^j{n+i \choose i-2j-1}s_{n+2j+1}\ \forall n\in\Bbb N,\forall i\in\Bbb Z^+$$
\end{proposition}

Now, if $T_{\frak S_n,v}(X)=\displaystyle\sum_{i=0}^{n-1}a_{n,i}X^i$, then the $a_{n,i}$ can be obtained from the previous proposition and  the facts that $a_{n,0}=s_n$ and that, from the definition of companion polynomial, we have that $$a_{n,i}=\frac{f(n,i)}{i!}-\sum_{k=0}^{i-1}\frac{a_{n,k}}{(i-k)!}\ \forall i\geq 1.$$ We will show in the following proposition the values for $i\leq 5$:

\begin{proposition}\label{ani05}

$a_{n,0}=s_n$\newline

$a_{n,1}=s_{n+1}-s_n$\newline

$a_{n,2}=\frac{1}{2} \left(s_n+n s_{n+1}-s_{n+2}\right)$\newline

$a_{n,3}=\frac{1}{12} \left(\left(n^2-n\right) s_{n+1}-2 s_n+6 s_{n+2}-2 s_{n+3}\right)$\newline

$a_{n,4}=\frac{1}{144} \left(\left(n^3-3 n^2+2 n\right) s_{n+1}+6 s_n-36 s_{n+2}-6 n s_{n+3}+6 s_{n+4}\right)$\newline

$a_{n,5}=\frac{1}{2880} \left(\left(12 n-12 n^2\right) s_{n+3}+\left(n^4-6 n^3+11 n^2-6 n\right) s_{n+1}-24 s_n+240 s_{n+2}-120 s_{n+4}+24 s_{n+5}\right)$
\end{proposition}

Note that the independent terms $a_{n,0}$ are just the Euler zigzag numbers. It is obvious from the definition of companion polynomial that $i!a_{n,i}$ is always an integer. We will call to $$s_{n,i}=i!a_{n,i}$$ the generalized zigzag number of order $i$. We will give now the first $15$ ones for $i$ up to $5$:\newline

For $i=0:1,1,1,2,5,16,61,272,1385,7936,50521,353792,2702765,22368256,199360981$\newline

For $i=1:0,0,1,3,11,45,211,1113,6551,42585,303271,2348973,19665491,176992725,1704396331$\newline

For $i=2:0,0,0,1,8,49,308,2031,14352,108833,885676,7715951,71760856,710303697,7460451204$\newline

For $i=3:0,0,0,0,2,25,238,2100,18594,169431,1610158,16042070,167927762,1847548053,21351747870$\newline

For $i=4:0,0,0,0,0,5,96,1276,15104,172319,1967392,22887002,273977600,3392750041,43582823456$\newline

For $i=5:0,0,0,0,0,0,16,427,7600,115791,1654336,23112848,322993424,4573319530,66117354592$\newline 

It is an immediate consequence ot Theorem \ref{lcxidilag} that for every $n\in\Bbb N$ there exists $n$ rational numbers $g(n,0),\dots,g(n,n-1)$ such that \begin{equation}\label{ptsnv}T_{\frak S_n,v}(X)=\sum_{i=0}^{n-1} g(n,i)X^i\frac{d^i}{dX^i}L_{n-1}(-X)\end{equation}

The first $6$ polynomials can be obtained from Proposition \ref{ani05}. They are shown next along with the tuples $(g(n,0),\dots,g(n,n-1))$:

For $n=1$, $$1$$ $$(1)$$

For $n=2$, $$X+1$$

$$(1,0)$$

For $n=3$, $$\frac{1}{2}\,{X}^{2}+3\,X+2$$

$$(2,\frac{1}{2},0)$$

For $n=4$, $$\frac{1}{3}\,{X}^{3}+4\,{X}^{2}+11\,X+5$$

$$(5,\frac{4}{3},\frac{1}{6},0)$$

For $n=5$, $${\frac {5\,{X}^{4}}{24}}+{\frac {25\,{X}^{3}}{6}}+{\frac {49\,{X}^{2}}{2}}+45\,X+16$$

$$(16,{\frac {19}{4}},\frac{5}{6},\frac{1}{12},0)$$

For $n=6$, $$\frac{2}{15}\,{X}^{5}+4\,{X}^{4}+{\frac {119\,{X}^{3}}{3}}+154\,{X}^{2}+211\,X+61
$$

$$(61,{\frac {94}{5}},{\frac {37}{10}},\frac{1}{2},\frac{1}{24},0)$$

Now we will analyze their associated differential equations. We conclude from (\ref{ptsnv}) and (\ref{dxilnx}) that $T_{\frak S_n,v}(-X)=\sum_{i=0}^{n-1} g(n,i)(-X)^i\frac{d^i}{dX^i}L_{n-1}(X)=$

$\sum_{i=0}^{n-1}g(n,i)(-X)^i (A_{n-1,i}(X) L_{n-1}(X)+B_{n-1,i}(X)\frac{d}{d X}L_{n-1}(X))=Q_n(X) L_{n-1}(X)+R_n(X) \frac{d}{d X}L_{n-1}(X)$, where $$Q_n(X)=\sum_{i=0}^{n-1}g(n,i)(-X)^i A_{n-1,i}(X)$$ and $$R_n(X)=\sum_{i=0}^{n-1}g(n,i)(-X)^i B_{n-1,i}(X).$$  Now, since $(-X)^i A_{n-1,i}(X)$ and $(-X)^i B_{n-1,i}(X)$ are polynomials, then the Laguerrean differential equation $\frak L(T_{\frak S_n,v}(-X))$ is of the form $$P_{1,n}(X)\frac{d^2 Y}{d X^2}+P_{2,n}(X)\frac{d Y}{d X}+P_{3,n}(X) Y(X)=0,$$ where $P_{1,n}(X),P_{2,n}(X),P_{3,n}(X)$ are polynomials. Finally, the differential equation associated to $y(X)=T_{\frak S_n,v}(X)$ is $$P_{1,n}(-X)\frac{d^2 Y}{d X^2}-P_{2,n}(-X)\frac{d Y}{d X}+P_{3,n}(-X) Y(X)=0.$$

The $Q_n(X),R_n(X)$ and the differential equations satisfied by the $T_{\frak S_n,v}(X)$ (after simplifying $X$) are shown next For $n$ from $2$ to $6$:\newline

For $n=2$, $$Q_2(X)=1,R_2(X)=0$$

$X Y''+(X+1) Y'- Y=0$\newline

For $n=3$, $$Q_3(X)=2,R_3(X)=\frac{-1}{2}\,X$$

$\frac{1}{2} (X+8) X Y''+\frac{1}{2} \left(X^2+8 X+8\right) Y'-(X+6) Y=0$\newline

For $n=4$, $$Q_4(X)=5-\frac{1}{2}\,X,R_4(X)=\frac{-3}{2}\,X+\frac{1}{6}\,{X}^{2}$$

$\frac{1}{3} \left(X^2+17 X+75\right) X Y''+\frac{1}{3} \left(X^3+16 X^2+75 X+75\right) Y'-\left(X^2+14 X+55\right) Y=0$\newline

For $n=5$, $$Q_5(X)=16-4\,X+\frac{1}{3}\,{X}^{2},R_5(X)={\frac {-23}{4}}\,X+\frac{4}{3}\,{X}^{2}-\frac{1}{12}\,{X}^{3}$$

$\frac{1}{12} (5 X^3+130 X^2+1073 X+3072) X Y''+\frac{1}{12} (5 X^4+120 X^3+943 X^2+3072 X+3072) Y'-$

$\frac{1}{3} (5 X^3+110 X^2+789 X+2160) Y=0$\newline

For $n=6$, $$Q_6(X)=61-{\frac {99}{4}}\,X+{\frac {25}{6}}\,{X}^{2}-{\frac {5}{24}}\,{X}^{3
}
,R_6(X)=-{\frac {95}{4}}\,X+{\frac {497}{60}}\,{X}^{2}-{\frac {25}{24}}\,{X}^{
3}+\frac{1}{24}\,{X}^{4}
$$

$\frac{1}{15} (10 X^4+370 X^3+4763 X^2+25957 X+55815) X Y''+$

$\frac{1}{15} (10 X^5+340 X^4+4023 X^3+21194 X^2+55815 X+55815) Y'-$

$\frac{1}{3} (10 X^4+320 X^3+3593 X^2+17778 X+38613) Y=0$\newline

We can see that the situation of the differential equation associated to the staircase digraph is completely different to the ones obtained from dispositional digraphs with two lines, for which the companion polynomial can be decomposed in terms of a few number of derivatives of Laguerre polynomials and, in consequence, the degrees of the coefficients of $Y^{\prime\prime},Y^{\prime}$ and $Y$ are low.

\section{Non-strict dispositions in digraphs}\label{nsdd}

In this section we will consider dispositions and counters in which the inequalities need not to be strict. In this case the transitive clausure of the digraphs needs not to define an order relation in order the counter to be non-zero, and we have an additional parameter indicating the cardinality of the set of assigned values.

\begin{definition} Let $\frak G=(V,A)$ be a directed graph, and let $i$ be a natural number. A non-strict disposition of size $i$ in $\frak G$ is a mapping $f:V\longrightarrow\{1,\dots,i\}$ such that $f(v_1)\geq f(v_2)$ whenever $(v_1,v_2)\in A$.
\end{definition}

If $\frak G$ and $i$ are as above, then the set of non-strict dispositions of size $i$ will be denoted by $\Sigma^{ns}_i(\frak G)$, and its cardinality by $\sigma^{ns}_i(\frak G)$. When $i$ is the order of the digraph we will write just $\Sigma^{ns}(\frak G)$ and $\sigma^{ns}(\frak G)$. The number $\sigma^{ns}(\frak G)$ will be called the non-strict counter of $\frak G$.

We will call non-strict Young tableaux over the digraph $\frak G$ to the triplet $$(\frak G,\Sigma^{ns}(\frak G),\sigma^{ns}(\frak G)).$$

\begin{theorem}\label{ns1} Let $\frak G$ be a digraph and let $i\in\Bbb N$. If $\frak R$ is the equivalence relation given in Proposition \ref{eredcw} and $\overline{\frak G}$ is the quotient digraph $\frak G/\frak R$, then $$\sigma^{ns}_i(\frak G)=\sigma^{ns}_i(\frak{\overline G})$$
\end{theorem}

\begin{proof} If $u,v\in V$ and $u,v$ are in a closed walk $z_1,z_2,\dots,z_n$, and if $f:V\longrightarrow\{1,\dots ,i\}$ is a non-strict disposition of size $i$ in $\frak G$, then $$f(z_1)\geq f(z_j)\geq f(z_1)\ \forall j\in\{1,\dots,n\},$$ and therefore $f(u)=f(v)$. So, the mapping $\overline{f}:\overline{V}\longrightarrow\{1,\dots ,i\}$ defined by $\overline{f}(\overline{v})=f(v)$ is well-defined, because the value $f(v)$ does not depends on the representant in the equivalence class, and it is a non-strict disposition of size $i$ in $\overline{\frak G}$.

Reciprocally, if $$\overline{f}:\overline{V}\longrightarrow\{1,\dots ,i\}$$ is a non-strict disposition of size $i$ in $\overline{\frak G}$, then we will prove that the mapping $$f:V\longrightarrow\{1,\dots ,i\}$$ defined by $f(v)=\overline{f}(\overline{v})$ is a non-strict disposition of size $i$ in $\frak G$. If $(u,v)\in A$, then $(\overline{u},\overline{v})\in\overline{A}$ by definition of quotient digraph, and $\overline{f}(\overline{u})\geq\overline{f}(\overline{v})$ because $\overline{f}$ is a disposition, and hence $f(u)\geq f(v)$.

Now, it is easy to prove that, if $\frak D$ and $\overline{\frak D}$ are the sets of non-strict dispositions of size $i$ in $\frak G$ and $\overline{\frak G}$, respectively, then the mappings $\delta_1:\frak D\longrightarrow\overline{\frak D}$ and $\delta_2:\overline{\frak D}\longrightarrow\frak D$ defined by $\delta_1(f)=\overline{f}$ and $\delta_2(\overline{f})=f$, respectively, are inverses one of each other.
\end{proof}

Hence, we can assume without loosing generality that the digraph has not directed cycles.

\begin{theorem}\label{nsbis} If $\frak G$ has not directed cycles, then $$\sigma^{ns}_i(\frak G)=\sum_{1\leq j\leq i}\ \ \sum_{\emptyset\subsetneq S\subseteq\text{min}(\frak G)}(-1)^{\vert S\vert+1}\sigma^{ns}_j(\frak G-S),$$ and $$\sigma^{ns}_i(\frak G)=\sum_{1\leq j\leq i}\ \ \sum_{\emptyset\subsetneq S\subseteq\text{max}(\frak G)}(-1)^{\vert S\vert+1}\sigma^{ns}_j(\frak G-S).$$
\end{theorem}

\begin{proof} It is an easy consequence of the inclusion-exclusion principle.
\end{proof}

The formulas in the previous theorem are the non-strict versions of the ones in Theorem \ref{rec}, and provide us the recurrence relations with which non-strict counters can be found.

\begin{theorem}\label{nstri} $$\sigma^{ns}_i(K_1)=i.$$
\end{theorem}

\begin{proof} If $K_1=(\{v\},\emptyset)$, then the dispositions of size $i$ in $K_1$ are the mappings $f_j$ with $j\in\{1,\dots,i\}$ defined by $f_j(v)=j$.
\end{proof}

Theorems \ref{ns1},\ref{nsbis} and \ref{nstri} let us to calculate the number of dispositions for any digraph. Moreover, if we convey, as we did for strict dispositions, that the counter of a digraph of order $0$ is $1$, then Theorem \ref{nstri} follows from Theorem \ref{nsbis}.

\begin{proposition}\label{snspn} For any natural number $m$, $$\sigma^{ns}_i(P_n)=\binom{i+n-1}{n}=CR_{i,n}.$$
\end{proposition}

\begin{proof} It is a classic problem in combinatorics. Beginning with a set with $n$ elements, we put $i-1$ separators so that the first separator separates the $1$s and $2$s, and so on.
\end{proof}

\begin{example} If we consider the digraph $\frak G$ in the following figure then by Theorem \ref{ns1}, $\sigma^{ns}_i(\frak G)=\sigma^{ns}_i(P_3)$, and by Proposition \ref{snspn}, this equals $\binom{i+2}{3}$.
\end{example}

\begin{figure}[H]
\caption{Directed path with directed cycle}
\includegraphics[width=5 in]{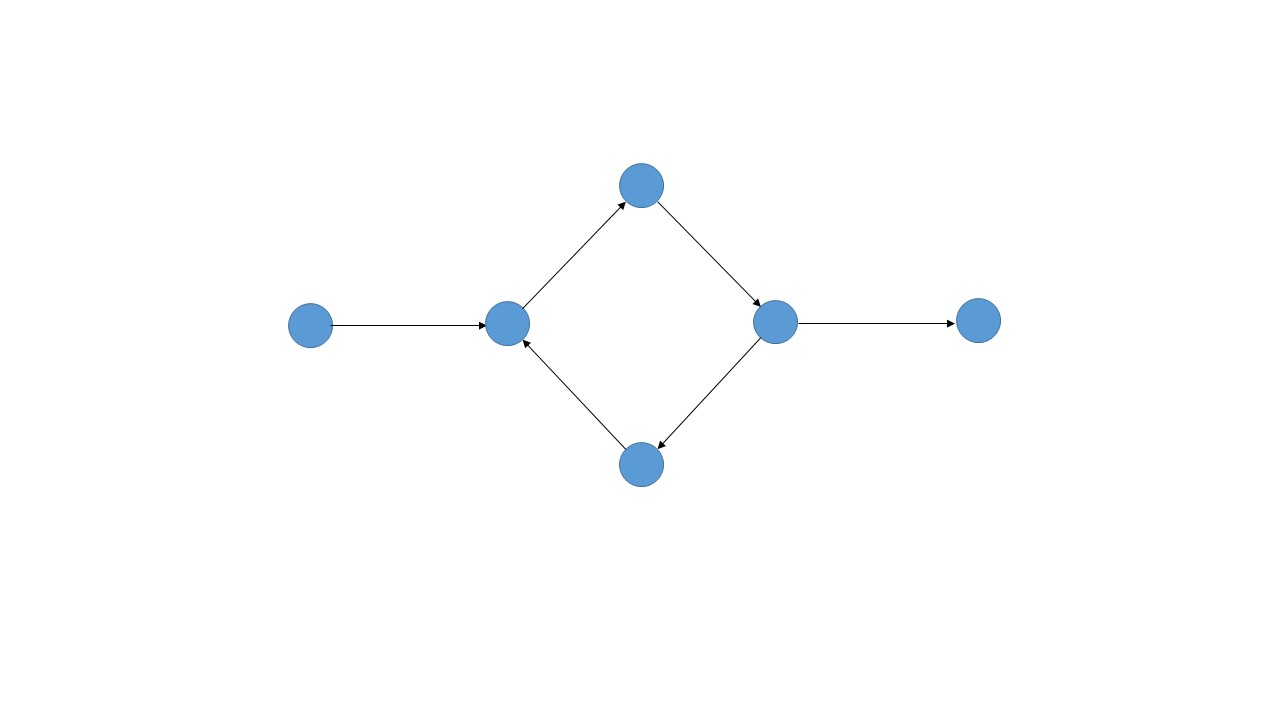}
\end{figure}

The following theorem is the non-strict analogous of Theorem \ref{contcc}. Its proof follows the line of that theorem, and will be omitted.

\begin{theorem}\label{nsg1gr} If $\frak G=(V,A)$ is a digraph of order $n$ and $i\in\Bbb N$, and $$\frak G_1=(V_1,A_1),\dots,\frak G_r=(V_r,A_r)$$ are the connected components of orders $n_1,\dots,n_r$, respectively, of the underlying undirected graph $\frak U(\frak G)$, then $$\sigma^{ns}_i(\frak G)=\sigma^{ns}_i(\frak G_1)\cdots \sigma^{ns}_i(\frak G_r).$$
\end{theorem}

\begin{corollary} For any natural numbers $n,i$, $$\sigma^{ns}_i(E_n)=i^n.$$
\end{corollary}

\begin{proof} Use the previous theorem and Theorem \ref{nstri}.
\end{proof}

We can define the non-strict generatrix function in a similar way as we did in the strict case, both for the case when the size of the dispositions is a fixed value and when it is equal to the number of vertices.

\begin{definition} Let $S=\{\frak G_i\}_{i\in\Bbb Z^+}$, where $\Bbb Z^+=\{i\in\Bbb Z\mid i\geq 0\}$, be a sequence of digraphs such that $\frak G_i$ is a subdigraph of $\frak G_{i+1}$ for every $i$, and let $j\in\Bbb N$. We will call non-strict generatrix function of size $j$ of $S$ to $$F^{ns}_{j,S}(X)=\sum_{i=0}^{\infty}\frac{\sigma^{ns}_j(\frak G_i)}{i!}X^i.$$
\end{definition}

\begin{definition} Let $S=\{\frak G_i\}_{i\in\Bbb Z^+}$ be a sequence of digraphs such that $\frak G_i$ is a subdigraph of $\frak G_{i+1}$ for every $i$. We will call non-strict generatrix function of $S$ to $$F^{ns}_S(X)=\sum_{i=0}^{\infty}\frac{\sigma^{ns}(\frak G_i)}{i!}X^i.$$
\end{definition}

\begin{proposition} If $S=\{P_n\}_{n\in\Bbb Z^+}$ is a sequence of directed paths and $j\in\Bbb N$, then $$F^{ns}_{j,S}(X)=\exp(X)L_{j-1}(-X),$$ where $L_{j-1}$ is the corresponding Laguerre polynomial.
\end{proposition}

We will use the following lemma, in which, of course, the combinatorial numbers in which the lower index is greater than the upper one are taken to be $0$:

\begin{lemma} If $a,b\in\Bbb Z^+$, then $$\sum_{s=0}^{\infty}{a \choose s}{b\choose s}={a+b \choose b}$$
\end{lemma}

\begin{proof} It is an immediate consequence of Chu-Vandermonde identity, which states that $$\sum_{s=0}^k {m \choose s} {n-m \choose k-s}={n \choose k},$$ by taking $m=a,n=a+b,k=b$.
\end{proof}

Now we will prove the proposition:

\begin{proof} On the one hand, by Proposition \ref{snspn}, $$F^{ns}_{j,S}(X)=\sum_{i=0}^{\infty}\frac{{j+i-1 \choose i}}{i!} X^i$$ On the other hand, $$\exp(X)L_{j-1}(-X)=\sum_{i=0}^{\infty}(\sum_{s=0}^{j-1}\frac{{j-1 \choose s}}{s!(i-s)!}) X^i,$$ and the result follows from the previous lemma.
\end{proof}

\begin{proposition}\label{mbffk} If $S=\{P_n\}_{n\in\Bbb Z^+}$ is a sequence of directed paths, then $$F^{ns}_S(X)=\frac{1}{2}(1+\exp(2X)I_{0}(2X)),$$ where $I_0$ is the modified Bessel function of the first kind.
\end{proposition}

We will need the next lemma:

\begin{lemma} If $i\in\Bbb Z^+$, then $$\frac{1}{2}\sum_{s=0}^{\infty}\frac{2^{i-2s}}{(i-2s)!(s!)^2}=\begin{cases}\frac{1}{2} & \text{if }i=0,\\ \frac{{2i-1 \choose i}}{i!} & \text{otherwise}\end{cases}$$
\end{lemma}

\begin{proof} It is obvious if $i=0$. Let us suppose that $i>0$. We have that \begin{equation}\label{opx2n1}(1+X)^{2i}=\sum_{k=0}^{2i} {2i \choose k} X^k,\end{equation} and also $$(1+X)^{2i}=(1+(2X+X^2))^i=\sum_{s_1=0}^i {i \choose s_1} (2X+X^2)^{s_1}$$ and, using the binomial theorem we obtain that the previous expression equals \begin{equation}\label{opx2n2}\sum_{s_1=0}^i\sum_{s_2=0}^{s_1} {i \choose s_1}{s_1 \choose s_2} 2^{s_1-s_2} X^{s_1+s_2}\end{equation} Be equating the coefficient of $X^i$ in (\ref{opx2n1}) and in (\ref{opx2n2}) and rearranging the indexes we get $${2i \choose i}=\sum_{s_2=0}^i {i \choose s_2}{i-s_2 \choose s_2} 2^{i-2s_2}$$ Since ${2i \choose i}=2{2i-1 \choose i}$, we conclude that $${2i-1 \choose i}=\sum_{s_2=0}^i {i \choose s_2}{i-s_2 \choose s_2} 2^{i-2s_2-1},$$ and the result follows trivially.
\end{proof}

Now we will prove the proposition:

\begin{proof} By Proposition \ref{snspn}, $$F^{ns}_S(X)=\sum_{i=0}^{\infty} \frac{{2i-1 \choose i}}{i!} X^i$$ On the other hand, $$\exp(2X)I_{0}(2X)=(\sum_{r=0}^{\infty} \frac{(2X)^r}{r!})(\sum_{s=0}^{\infty}(\frac{X^s}{s!})^2)=\sum_{i=0}^{\infty}(\sum_{s=0}^{\infty} \frac{2^{i-2s}}{(i-2s)!(s!)^2}) X^i,$$ and the result follows from the previous lemma.
\end{proof}

Next we will analyze the non-strict counters of some dispositional digraphs with two rows. Specifically, we will consider the same dispositional digraphs  $$\frak G_{n_1,n_2}=\mathcal G([n_2,0],[n_1,0])$$ with $n_1\leq n_2$ that we studied in Section \ref{disptwolines}. We will denote again the vertices by $u_1,\dots,u_{n_1},v_1,\dots,v_{n_2}$.

The following proposition is immediate:

\begin{proposition}\label{minnstr} $$\min (\frak G_{n_1,n_2})=\begin{cases} \{u_{n_1},v_{n_2}\} & \text{ if }n_1<n_2 \\ \{u_{n_1}\} & \text{ if }n_1=n_2\end{cases}$$
\end{proposition}

\begin{proposition} $\sigma^{ns}_i (\frak G_{n_1,n_2})$ is $$\displaystyle\sum_{1\leq j\leq i}[\sigma^{ns}_j (\frak G_{n_1-1,n_2})+\sigma^{ns}_j (\frak G_{n_1,n_2-1})-\sigma^{ns}_j (\frak G_{n_1-1,n_2-1})]$$ when $n_1<n_2$ and is $$\displaystyle\sum_{1\leq j\leq i}\sigma^{ns}_j (\frak G_{n_1-1,n_1})$$ when $n_1=n_2$, with the initial conditions $$\sigma^{ns}_j (\frak G_{n_1,0})=\sigma^{ns}_j (P_{n_1})=PR_{j,n_1}$$ and $$\sigma^{ns}_j (\frak G_{0,n_2})=\sigma^{ns}_j (P_{n_2})=PR_{j,n_2}$$
\end{proposition}

\begin{proof} Use Theorem \ref{nsbis} and Propositions \ref{snspn} and \ref{minnstr}.
\end{proof}

The following result can now be easily proved inductively by using the previous proposition, and the proof will be omitted.

\begin{proposition} If $n_1\leq n_2$, then $$\sigma^{ns}_i(\frak G_{n_1,n_2})=\frac{in_2-(n_1-1)i+n_1}{(n_2+1)i}CR_{i,n_1}CR_{i,n_2}=(1+\frac{n_1(1-i)}{(n_2+1)i})CR_{i,n_1}CR_{i,n_2}$$
\end{proposition}

To finish the paper, we will present some differential equations derived from non-strict dispositions in digraphs.

\begin{proposition} The generatrix function $F^{ns}_S(X)$ of Proposition \ref{mbffk} satisfies the differential equation $$XY^{\prime\prime}+(1-4x)Y^{\prime}-2Y+1=0$$
\end{proposition}

\begin{proof} It is an easy consequence of that proposition.
\end{proof}

As a consequence, we obtain:

\begin{corollary}\label{deexp22} The function $\exp(2x)I_0(2x)$ satisfies the differential equation $$XY^{\prime\prime}+(1-4x)Y^{\prime}-2Y=0$$
\end{corollary}

We note that the previous differential equation is formally very similar to Laguerre's equation, although the solution is not a polynomial, but it has a discrete interpretation in terms of non-strict counters of digraphs.

In general, if we define for $n,m,k\in\Bbb N$ $$F_{n,m,k}(X)=\exp(nx)I_m(kx),$$ then it is easy to prove, by expanding the function as a power series, the following:

\begin{proposition} The function $F_{n,m,k}(X)$ satisfies the differential equation $$X^2Y^{\prime\prime}+(X-2n X^2)Y^{\prime}+(-m^2-nX+(n^2-k^2)X^2)Y=0$$
\end{proposition}

In particular, when $n=k$ the $X^2$ in the coefficient of $Y$ vanishes, and we get that $F_{n,k,n}(X)$ satisfies the differential equation $$X^2Y^{\prime\prime}+(X-2n X^2)Y^{\prime}-(m^2+nX)Y=0$$

If, besides, $m=0$ the $X$ factor simplifies, and we get the differential equation $$XY^{\prime\prime}+(1-2nX)Y^{\prime}-nY=0,$$ which for $n=2$ is the one shown in Corollary \ref{deexp22}.

It would be interesting to study if the functions $F_{n,m,k}(X)$ have a combinatorial interpretation similar to the one that we presented in the corollary.

{\bf Acknowledgements}

We would like to thank Martin Milani\v{c} for helpful discussions.

\end{document}